\documentclass{article}

\usepackage{blindtext, xcolor}

\usepackage[centertags]{amsmath}
\usepackage{hyperref}
\usepackage{amsfonts}
\usepackage{amssymb}
\usepackage{amsthm}
\usepackage{newlfont}
\usepackage{amscd}
\usepackage{amsmath,amscd}
\usepackage{graphicx}
\usepackage{ tipa }
\usepackage{mathrsfs}
\usepackage{comment}

\usepackage[all]{xy}
\usepackage{verbatim}
\usepackage{diagrams}

\usepackage{color}

\newcommand{\CC}{\mathbb{C}}
\newcommand{\SSS}{\mathbb{S}}
\newcommand{\QQ}{\mathbb{Q}}
\newcommand{\KK}{\mathbb{K}}
\newcommand{\ZZ}{\mathbb{Z}}

\newcommand{\RR}{\mathbb{R}}

\newcommand{\tC}{\mathrm{C}}

\newcommand{\cC}{\mathcal{C}}
\newcommand{\cD}{\mathcal{D}}
\newcommand{\cE}{\mathcal{E}}
\newcommand{\cF}{\mathcal{F}}

\newcommand{\cL}{\mathcal{L}}
\newcommand{\cM}{\mathcal{M}}
\newcommand{\cN}{\mathcal{N}}

\newcommand{\cS}{\mathcal{S}}

\newcommand{\cW}{\mathcal{W}}

\newcommand{\ku}{\mathtt{ku}}

\newcommand{\subq}[3]{\SSS^{#1, #2}_{#3}}

\def\Csep{\mathtt{SC^*}}

\def\FCW{\mathtt{CW}^f_*}
\def\CW{\mathtt{CW}_*}
\def\NC{\mathtt{NC}}
\def\NCW{\mathtt{NCW}}
\def\Sp{\mathtt{Sp}}
\def\SpM{\mathtt{Sp^M}}
\def\Cat{\mathtt{Cat}}

\def\NSp{\mathtt{NSp}}

\def\FNCW{\mathtt{NCW}^f}

\newcommand{\Fin}{\mathtt{Fin}_*}

\def\Epi{\mathtt{Epi}}

\def\Finsk{\mathtt{Fin^{\leq k}_*}}

\def\Fins1{\mathtt{Fin^{\leq 1}_*}}

\def\Gred{\mathfrak{G}}

\def\Nat{\mathrm{Nat}}

\def\Chi{\mathcal{X}}

\newtheorem{thm}{Theorem}[section]

\newtheorem{cor}[thm]{Corollary}

\newtheorem{lem}[thm]{Lemma}
\newtheorem{prop}[thm]{Proposition}

\newtheorem{example}{Example}
\theoremstyle{definition}
\newtheorem{define}[thm]{Definition}
\theoremstyle{remark}
\newtheorem{rem}[thm]{Remark}

\DeclareMathOperator{\im}{Im}

\DeclareMathOperator{\Fun}{Fun}

\DeclareMathOperator{\op}{op}

\DeclareMathOperator{\Inj}{Inj}

\DeclareMathOperator{\Map}{Map}

\DeclareMathOperator{\Hom}{Hom}

\DeclareMathOperator{\prehocolim}{hocolim}
\DeclareMathOperator{\precolim}{colim}

\DeclareMathOperator{\uTop}{Top_u}
\DeclareMathOperator{\Top}{Top}

\def\colim{\mathop{\precolim}}
\def\hocolim{\mathop{\prehocolim}}

\def\End{\textrm{End}}
\def\im{\textrm{Im\,}}

\DeclareFontEncoding{OT2}{}{} 
\DeclareTextFontCommand{\textcyr}{\fontencoding{OT2}\fontfamily{wncyr}\fontseries{m}\fontshape{n}\selectfont}

\newcommand\noloc{%
  \nobreak
  \mspace{6mu plus 1mu}
  {:}
  \nonscript\mkern-\thinmuskip
  \mathpunct{}
  \mspace{2mu}
}

\begin{document}
\title{Suspension spectra of matrix algebras, the rank filtration, and rational noncommutative CW-spectra}

\author{Gregory Arone \thanks{Supported in part by the Swedish Research Council, grant number 2016-05440} \\ Stockholm University \\ gregory.arone@math.su.se \and Ilan Barnea \thanks{Supported by ISF 786/19} \\ Haifa University \\ ilanbarnea770@gmail.com \and Tomer M. Schlank \thanks{Supported by ISF 1588/18 and BSF 2018389} \\ Hebrew University \\tomer.schlank@gmail.com}

\maketitle

\begin{abstract} In a companion paper~\cite{ABS1} we introduced the stable $\infty$-category of noncommutative CW-spectra, which we denoted $\NSp$.  Let $\cM$ denote the full spectrally enriched subcategory of $\NSp$ whose objects are the non-commutative suspension spectra of matrix algebras. In~\cite{ABS1} we proved that $\NSp$ is equivalent to the $\infty$-category of spectral presheaves on $\cM$.  In this paper we investigate the structure of $\cM$, and derive some consequences regarding the structure of $\NSp$.

To begin with, we introduce a rank filtration of $\cM$. We show that the mapping spectra of $\cM$ map naturally to the connective $K$-theory spectrum $ku$, and that the rank filtration of $\cM$ is a lift of the classical rank filtration of $ku$. We describe the subquotients of the rank filtration in terms of complexes of direct-sum decompositions which also arose in the study of $K$-theory and of Weiss's orthogonal calculus. We prove that the rank filtration stabilizes rationally after the first stage. Using this we give an explicit model of the rationalization of $\NSp$ as presheaves of rational spectra on the category of finite-dimensional Hilbert spaces and unitary transformations up to scaling.
Our results also have consequences for the $p$-localization and the chromatic localization of $\cM$.
\end{abstract}

\tableofcontents

\section{Introduction}
In our previous paper~\cite{ABS1} we introduced the $\infty$-category of noncommutative CW-spectra, which we denoted $\NSp$. This is the stabilization of the $\infty$-category of noncommutative CW-complexes, denoted $\NCW$, which in turn is the ind completion of the $\infty$-category of finite noncommutative CW-complexes, denoted $\FNCW$. The latter is defined as \textbf{opposite} of the topological nerve of the topological category whose objects are the $C^*$-algebras which are noncommutative CW-complexes in the sense of \cite{ELP} and whose hom-spaces are given by taking the topology of pointwise norm convergence on the sets of $*$-homomorphisms.

The main result of \cite{ABS1} says that $\NSp$ is equivalent to the $\infty$-category of spectral presheaves over a small spectrally enriched $\infty$-category $\cM$. The spectral $\infty$-category $\cM$ is defined to be the full spectral subcategory of $\NSp$, whose objects are noncommutative suspension spectra of matrix algebras. In this paper we analyze the category $\cM$ in considerable detail. We introduce a rank filtration of $\cM$, describe the subquotients of the rank filtration and use this to give an explicit model for the rationalization of $\NSp$.

\begin{rem}
  In \cite{ABS1} we made extensive use of Hinich's theory of enriched $\infty$-categories (see \cite{Hin2,Hin3}). In this paper we also use this theory and terminology on a few occasions. The interested reader is referred also to \cite[Section 3]{ABS1} for a summery of the parts of this theory relevant to us.
\end{rem}

Given an integer $k\geq 1$, let $M_k$ be the $C^*$-algebra of $k\times k$ matrices over $\CC$. We have the stabilization functor
$$\Sigma^\infty_{\NC}:\NCW\to\NSp$$
and the set of objects of $\cM$ is $\{\Sigma^\infty_{\NC}M_k|\,k\geq 1\}$. Thus the objects of $\cM$ are in one to one correspondence with the positive integers. Given two integers $k, l$, we denote the spectral mapping object in $\cM$ by
$$\SSS^{k,l}:=\Hom_{\cM}(\Sigma^\infty_{\NC} M_k,\Sigma^\infty_{\NC} M_l)\simeq \Hom_{\NSp}(\Sigma^\infty_{\NC} M_k,\Sigma^\infty_{\NC} M_l)\in\Sp,$$
where $\Sp$ is the $\infty$-category of (ordinary) spectra.

Our goal in this paper is to analyze the category $\cM$. For this, it will be convenient to use a model category that models the $\infty$-category of spectra $\Sp$. Analyzing $\cM$ means, firstly, that we want to describe, for each $k$ and $l$, the homotopy type of $\SSS^{k,l}$. For this purpose it is adequate to use the simple model for spectra, as a sequence of pointed spaces with structure maps. But we also want to model the composition maps,
\begin{equation}\label{eq: stable composition}
\SSS^{k,l}\wedge \SSS^{j,k}\to \SSS^{j,l}.
\end{equation}
To do this properly, we need to use a more sophisticated model for spectra, which incorporates a smash product.

To be more explicit, in \cite{ABS1} we defined a ``strict" model of $\cM$. That is, we defined a category $\cM_s$ strictly enriched in a certain monoidal model category of spectra $\SpM$, such that the enriched $\infty$-localization of $\cM_s$ is equivalent to $\cM$.
The model category $\SpM$ is the category of pointed continuous functors from pointed finite CW-complexes to pointed topological spaces, endowed with the stable model structure. (By a topological space in this paper we always mean a compactly generated weak Hausdorff space.)
Day convolution turns $\SpM$ into a symmetric monoidal model category. See~\cite{Lyd, MMSS, Lyd1} for more details on this model structure.

Recall that any relative category, that is a pair $(\cC,\cW)$ consisting of a category $\cC$ an a subcategory $\cW\subseteq\cC$, has a canonically associated $\infty$-category $\cC_\infty$, obtained by formally inverting the morphisms in $\cW$ in the infinity categorical sense. There is also a canonical localization functor $\cC\to\cC_\infty$ satisfying a universal property. We refer the reader to \cite{Hin1} for a thorough account, and also to the discussion in \cite[Section 2.2]{BHH}. We refer to $\cC_\infty$ as the $\infty$-localization of $\cC$ (with respect to $\cW$). If $\cC$ is a model category or a (co)fibration category, we always take $\cW$ to be the set of weak equivalences in $\cC$.

We have a canonical equivalence of symmetric monoidal $\infty$-categories
$$\SpM_\infty\simeq\Sp.$$
We will identify the two $\infty$-categories above through this equivalence.
Similarly, if $\Top$ is the category of pointed topological spaces endowed with the Quillen model structure \cite{Qui}, then we have a canonical equivalence
$$\Top_\infty\simeq\cS_*$$
where $\cS_*$ is the $\infty$-category of pointed spaces. Again, we identify the two $\infty$-categories above through this equivalence.

We denote the localization functor from $\SpM$ to $\SpM_\infty=\Sp$ by $\partial_1$
$$\partial_1:\SpM\to \Sp.$$
If $G\in\SpM$ is a pointed continuous functors from pointed finite CW-complexes to pointed topological spaces, then $\partial_1 G$ is the spectrum corresponding to the sequence of spaces $\{G(S^0), G(S^1), \ldots $\} (where we identify a pointed topological space with its image in $\Top_\infty=\cS_*$). In other words we can write
$$\partial_1 G\simeq
{\hocolim}_n\Sigma^{-n}\Sigma^{\infty}G(S^n),$$
where by $\hocolim$ here we mean $\infty$-colimit in $\Sp$. This is known as the stabilization, or the first derivative of the functor $G$.


The way we define $\cM_s$ in \cite{ABS1} is by letting for every $k, l$
$$\Hom_{\cM_s}(M_k,M_l):=G_{k,l}\in\SpM,$$
where
\[
G_{k,l}(X):=\Csep(\tC_0(X, M_l),M_k)
\]
(see also Definition \ref{def: cM}). Here $\tC_0(X, M_l)$ is the space of pointed map from $X$ to $M_l$, considered as a $C^*$-algebra, and $\Csep(-,-)$ denotes the space of $C^*$-algebra maps with the topology of pointwise norm convergence. Since $\cM_s$ is a model for $\cM$ we have that $G_{k,l}$ is a model for $\SSS^{k,l}$, or in other words, the stabilization of $G_{k,l}$ is $\SSS^{k,l}$:
\[
\SSS^{k,l} \simeq\partial_1G_{k, l}\simeq{\hocolim}_n\Sigma^{-n}\Sigma^{\infty}G_{k, l}(S^n).
\]

\begin{rem}\label{r:M}
  In this paper we use both $\cM$ and $\cM_s$. For convenience of notation we denote both categories by $\cM$ trusting that it is clear from the context which is meant.
\end{rem}

We proceed to investigate the homotopy type of $\SSS^{k,l}$. It is not hard to show that if $l>k$ then $\SSS^{k,l}$ is contractible (see corollary~\ref{cor: l le k}), so we generally assume that $k\ge l$. To analyze $\SSS^{k,l}$ further, we introduce a natural filtration of $\cM$ in Section~\ref{section: stable rank}, which we call the {\it rank filtration}. More precisely, for each $k$ and $l$, we define a sequence of subfunctors $G_{k, l, i}\in\SpM$
\[
G_{k, l, 1}\subset G_{k, l, 2}\subset \cdots \subset G_{k, l, \lfloor\frac{k}{l}\rfloor}= G_{k, l, \lfloor\frac{k}{l}\rfloor+1}=\cdots =G_{k,l}.
\]
Upon stabilization, we obtain a sequence of spectra
\[
\SSS^{k,l,1} \hookrightarrow \SSS^{k,l,2} \hookrightarrow \cdots \hookrightarrow \SSS^{k,l,\lfloor\frac{k}{l}\rfloor}=\SSS^{k,l}.
\]
We think of this sequence as defining a filtration of $\SSS^{k,l}$. This filtration is multiplicatively compatible with composition. This means that there are maps, compatible with~\eqref{eq: stable composition}
\[
\SSS^{k,l,m}\wedge \SSS^{j,k,n}\to \SSS^{j,l,mn}.
\]
One can say that the category $\cM$ is filtered by the multiplicative monoid of positive integers.

The functors $G_{k,l,m}$ are, by definition, functors from the category of pointed finite CW-complexes to pointed topological spaces. But it turns out that they are determined by their restriction to the category of pointed finite sets. Specifically, we show in Theorem \ref{t:polynom inverse}
that $G_{k,l,m}$ is equivalent to both the strict and the derived left Kan extension of its restriction to pointed finite sets.
In other words, $G_{k,l,m}$ are $\Gamma$-spaces.

In Section~\ref{ss: finite} we describe explicitly the restriction of $G_{k, l}$ to pointed finite sets (Proposition~\ref{prop:SC inj Mlt Mk}). In Section~\ref{section: cofibrant} we show that the $\Gamma$-spaces $G_{k,l,m}$ are in fact cofibrant in a generalized Reedy model structure considered by Bousfield-Friedlander~\cite{BF}, Lydakis~\cite{Lyd1} and Berger-Moerdijk~\cite{BM}.

Proposition~\ref{prop:SC inj Mlt Mk} indicates that $G_{k,l}$ is similar to the $\Gamma$-space that models the K-theory spectrum $ku$. This is not a coincidence. The inclusion of matrix algebras $M_k\to M_{k+1}$ that sends a matrix $a$ to $a\oplus 0$ indues a natural transformation $G_{k,l}\to G_{k+1,l}$, which in turn induces a map of spectra $\SSS^{k,l}\to \SSS^{k+1, l}$. We prove in Section~\ref{section: k-theory} that for each fixed $l$, $\hocolim_{k\to\infty} \SSS^{k,l}\simeq ku$ (thus all the mapping spectra $\SSS^{k,l}$ are naturally spectra over $ku$). The case $l=1$ of this observation goes back to Segal~\cite{Se2} and was exploited extensively by Dadarlat with various collaborators (for example see~\cite{DM}). Furthermore, it turns out that the functor $l\mapsto \hocolim_{k\to \infty} \SSS^{k, l}$ takes values in module spectra over $ku$. This allows for a natural way do define bivariant connective $K$-theory on $\NSp$, which generalizes the definition of Dadarlat and McClure \cite{DM} to the noncommutative setting. We also observe in this section that the rank filtration of $\cM$ is a lift of the classical rank filtration of $ku$. Later in the paper we prove that the map $\SSS^{k,l}\to ku$ induces an isomorphism on $\pi_0$ (Lemma~\ref{lem: over ku}).

The results of Sections~\ref{ss: finite} and~\ref{section: cofibrant} are useful for the homotopic analysis of $\SSS^{k,l}$ and of the rank filtration. One of our main results is an explicit description of the associated graded filtration. Thus we describe, for each $k,l,m$, the homotopy cofiber spectrum $\subq{k}{l}{m}:= \SSS^{k,l,m}/\SSS^{k,l,m-1}$, as well as the induced maps
\[
\subq{k}{l}{m}\wedge \subq{j}{k}{n}\to \subq{j}{l}{mn}.
\]
Our description features certain $U(m)$-complexes $\cL_m^\diamond$. Roughly speaking $\cL_m^\diamond$ is the space of direct-sum decompositions of $\CC^m$ (see Definition~\ref{definition: decompositions}). These complexes have some remarkable homotopical properties, which we will review below. But first let us state the result describing the rank filtration in terms of $\cL_m^\diamond$. When $U$ and $V$  are unitary vector spaces, let $\Inj(U,V)$ denote the space of (necessarily injective) linear transformations of $U$ into $V$ that preserve the unitary product. Note for future reference that there is a homeomorphism $\Inj(\CC^m, \CC^n)\cong U(n)/U(n-m)$. The following is Theorem \ref{theorem: main} in the text:
\begin{thm}\label{thm: intro main}
There is an equivalence of spectra
\[
\subq{k}{l}{m}\simeq  \Sigma^\infty \cL_m^\diamond \wedge_{U(m)} \Inj(\CC^{lm}, \CC^k)_+,
\]
where $U(m)$ acts through the identification $\CC^{lm}\cong \CC^{m}\otimes \CC^l$. The composition map $\subq{k}{l}{n}\wedge \subq{j}{k}{m}\to \subq{j}{l}{mn}$ is
determined by the map $\cL_n^\diamond \wedge \cL_m^\diamond \to \cL_{mn}^\diamond$ defined by tensor product of decompositions, and the obvious composition map
\[
\Inj(\CC^{ln}, \CC^k)\times  \Inj(\CC^{km}, \CC^j) \to \Inj(\CC^{lmn}, \CC^{km})\times  \Inj(\CC^{km}, \CC^j) \to \Inj(\CC^{lmn}, \CC^j).
\]
\end{thm}

The complexes $\cL^\diamond_m$ were first introduced in~\cite{Ar}, and were studied in detail in~\cite{Banff} and~\cite{AL}. They play a role in describing the subquotients of the rank filtration of $K$-theory~\cite{AL-Crelle, AL-Fundamenta}. Therefore it is perhaps not surprising that they play a similar role in the rank filtration of $\cM$, given the connection between $\cM$ and $\ku$.

Next proposition lists some relevant facts about the complexes $\cL_m^\diamond$ (Proposition~\ref{proposition: L_m facts} in the paper).
\begin{prop} \label{prop: Lm facts intro}
\begin{enumerate}
\item $\cL^\diamond_1=S^0$.
\item \label{intro rational} The complex $\cL^\diamond_m$ is rationally contractible for all $m>1$.
\item \label{prime powers} The complex $\cL^\diamond_m$ is contractible unless $m$ is a prime power.
\item \label{p-local} If $m=p^k$ where $p$ is a prime and $k>0$ then $\cL^\diamond_{p^k}$ is $p$-local.
\item The complex $\cL^\diamond_{p^k}$ has chromatic type $k$.
\end{enumerate}
\end{prop}
Here are some consequences of Theorem~\ref{thm: intro main} and Proposition~\ref{prop: Lm facts intro}. To begin with, we have a simple description of the endomorphisms in $\cM$. The endomorphism spectrum of $\Sigma^\infty_{\NC}M_k$ is the group ring spectrum of the projective unitary group $PU(k)$.
\begin{cor}
$\SSS^{k,k}\simeq \Sigma^\infty PU(k)_+.$
\end{cor}

But our main application of Theorem~\ref{thm: intro main} and Proposition~\ref{prop: Lm facts intro}(\ref{intro rational}) is to give a simplified description of the rational homotopy type of $\SSS^{k,l}$, and consequently of the rationalization of $\NSp$. Recall that the composition maps $\SSS^{j, k}\wedge \SSS^{k, l}\to \SSS^{j, l}$ restrict to maps of the form $\SSS^{j, k, 1}\wedge \SSS^{k, l, 1}\to \SSS^{j, l, 1}$. It follows that the spectra $ \SSS^{k, l, 1}$ assemble to a spectral category, that we denote $\cM^1$, which has the same objects as $\cM$, and equipped with a functor $\cM^1\to \cM$ that is the identity on objects. Informally speaking, $\cM^1$ is the first stage of the rank filtration of $\cM$.

It follows from Theorem~\ref{thm: intro main} that there is an equivalence
\[
\SSS^{k,l,1}=\SSS^{k,l}_1\simeq \Sigma^\infty {\Inj(\CC^l, \CC^k)/_{U(1)}}_+.
\]
It is worth noting that $\SSS^{k,l,1}$ is a {\it suspension spectrum}.
Let $\mathbb{P}\Inj$ be the topologically enriched symmetric  monoidal  category of finite positive dimensional  Hilbert spaces  and embeddings up to scalar. That is, up to isomorphism the objects of $\mathbb{P}\Inj$ are given by $\mathbb{C}^k$ for $k \geq 1$ and $$\mathbb{P}\Inj (\mathbb{C}^l , \mathbb{C}^k) = \Inj(\CC^l, \CC^k)/_{U(1)}=U(k)/ (U(1) \times U(k-l)). $$
$\mathbb{P}\Inj$ is a symmetric monoidal category, with the monoidal structure given by  the tensor product.
Since the functor $\Sigma^\infty_+,$ from topological spaces to our model of spectra $\SpM$, is symmetric monoidal, we can define a category enriched in $\SpM$, which we denote $\mathbb{P}\Inj^{\Sp}$, by applying $\Sigma^\infty_+$ to the mapping spaces of $\mathbb{P}\Inj$. We show in Section \ref{section: first stage} that
$$\cM^1\simeq(\mathbb{P}\Inj^{\Sp})^{\op}.$$

The following is an easy consequence of Proposition~\ref{prop: Lm facts intro}\eqref{intro rational}
\begin{cor}\label{cor: rational approximation}
The natural map
\[
\Sigma^\infty {\Inj(\CC^l, \CC^k)/_{U(1)}}_+\simeq\SSS^{k,l,1} \xrightarrow{\simeq_{\mathbb Q}} \SSS^{k,l}
\]
is a rational homotopy equivalence.
\end{cor}
\begin{rem}
If one lets $k$ go to $\infty$ in corollary~\ref{cor: rational approximation}, one obtains the classical fact that the canonical map $\Sigma^\infty \CC P^\infty_+ \to ku$ is a rational equivalence. So corollary~\ref{cor: rational approximation} can be thought of as a lift of this fact.
\end{rem}
Corollary~\ref{cor: rational approximation} says that the functor $\cM^1\to \cM$ is a rational equivalence of spectral categories. Using this, we can give a rather explicit description of the rationalization of $\NSp$.
We discuss the general construction of rational localization and $p$-localization of a stable, monoidal, $\infty$-category in Section~\ref{section: localizations}. Let $\NSp_{\mathbb{Q}}$ denote the rational localization of the $\infty$-category of noncommutative spectra $\NSp$ and let $\Sp_{\mathbb{Q}}$ denote the rational localization of the usual $\infty$-category of spectra $\Sp$ . It is well known that $\Sp_{\mathbb{Q}}$ is a symmetric monoidal presentable $\infty$-category, and the rationalization functor $$L_{\mathbb{Q}}\colon \Sp \to \Sp_{\mathbb{Q}}$$
is symmetric monoidal.

Let $\mathbb{P}\Inj_{\infty}$ denote the topological nerve of the topological category $\mathbb{P}\Inj$ defined above. Applying the symmetric monoidal functor
$$L_{\mathbb{Q}}\circ\Sigma^\infty_+:\cS\to \Sp_\QQ$$
to the mapping spaces of $\mathbb{P}\Inj_{\infty}$ we obtain an $\infty$-category enriched in $\Sp_\QQ$ which we denote by $\mathbb{P}\Inj_{\infty}^{\Sp_\QQ}$. Let $P_{\Sp_{\mathbb{Q}}} ((\mathbb{P}\Inj_{\infty}^{\Sp_\QQ})^{\op})$ denote the $\infty$-category of $\Sp_\QQ$-enriched functors from $\mathbb{P}\Inj_{\infty}^{\Sp_\QQ}$ to $\Sp_\QQ$. The following theorem summarizes the results of Section~\ref{section: localizations} about $\Sp_{\mathbb Q}$:

\begin{thm}[Theorem \ref{t:rational spectra}]\label{t:rational spectra0}
There are equivalences of symmetric monoidal $\infty$-categories
\[
\NSp_{\mathbb Q}\simeq P_{\Sp_{\mathbb{Q}}} ((\mathbb{P}\Inj_{\infty}^{\Sp_\QQ})^{\op})\simeq \mathrm{Fun}(\mathbb{P}\Inj_{\infty}, \Sp_{\mathbb{\QQ}}).
\]
\end{thm}

\begin{rem}
Note that the expression on the right in Theorem \ref{t:rational spectra0} does not use enriched $\infty$-categories. This is the usual $\infty$-category of functors from $\mathbb{P}\Inj_{\infty}$ to $\Sp_{\mathbb{\QQ}}$. Note also that the mapping spaces in $\mathbb{P}\Inj$ are all finite connected CW-complexes (manifolds, even). It is natural to wonder if one can given a more direct algebraic  model of $\NSp_{\mathbb{Q}}$ as a dg-category.

\end{rem}

\subsubsection*{$p$-local and chromatic picture}
Now instead of rationalizing, suppose we fix a prime $p$ and localize everything at $p$. One can obtain further information about the $p$-localization of $\cM$. It follows from Proposition~\ref{prop: Lm facts intro} parts~\eqref{prime powers} and~\eqref{p-local} that the filtration is $p$-locally constant except at powers of $p$. Therefore it is natural to regrade the filtration of $\SSS^{k,l}$ as follows
\[
\SSS^{k,l,1} \hookrightarrow \SSS^{k,l,p} \hookrightarrow \SSS^{k,l,p^2}\hookrightarrow \cdots \hookrightarrow \SSS^{k,l,p^i}\cdots
\]
With this grading, (the $p$-localization) of $\cM$ is a filtered category in the usual sense, that composition adds degrees. Furthermore, we have
\begin{cor}
Fix a prime $p$ and localize everything at $p$. The map
\[
\SSS^{k,l, p^n}\to \SSS^{k,l}
\]
induces an isomorphism on Morava $K(i)$-theory for $i\le n$.
\end{cor}
The last corollary may have consequences for ``noncommutative chromatic homotopy theory'', but we will not pursue it here.

\subsubsection*{Section by section outline of the paper}

In Section~\ref{ss: Glk} we set the stage by recalling some relevant definitions from~\cite{ABS1}. We introduce functors $G_{k,l}$, where $k,l$ are positive integers. The functors $G_{k,l}$ encode all the information about morphisms in $\cM$. More precisely, the stabilization of $G_{k,l}$ is the spectral mapping object from $k$ to $l$ in $\cM$. In Section~\ref{section: stable rank} we introduce a natural filtration of the functors $G_{k,l}$, which we call the rank filtration.

In Section~\ref{ss: finite} we give an explicit description of the restriction of $G_{k,l}$ to pointed finite sets. In Section~\ref{section: cofibrant} we establish various properties of the restriction of $G_{k,l}$ to finite sets. Most importantly, we show that the functor $G_{k,l}$ is a $\Gamma$-space, in the sense that it is determined by its values on finite sets, and we also observe that the restriction of $G_{k,l}$ to finite sets is cofibrant in the Reedy model structure on $\Gamma$-spaces.

In Section~\ref{section: k-theory} we show that all the mapping spectra in $\cM$ are equipped with a natural map to the connective $K$-theory spectrum $ku$. We observe that the rank filtration of $\cM$ is a lift of the classical rank filtration of $ku$. We also discuss how one can use our models to represent the $K$-theory functor on non-commutative complexes. We make connection with some work of Dadarlat and McClure~\cite{DM}.

In Sections~\ref{ss: subquotients} and~\ref{ss: Ln} we describe the subquotients of the rank filtration in terms of complexes of direct-sum decompositions that arose earlier in the study of the rank filtration of $ku$. Since complexes of direct-sum decompositions are well-studied, we obtain interesting consequences about $\cM$ and $\NSp$. In Section~\ref{ss: calculations} we use the results of preceding sections to calculate the mapping spectra in $\cM$ in some cases. In Section~\ref{section: localizations} we use those results to give an explicit model for the rationalization of $\cM$ and $\NSp$. Rationally, $\NSp$ is equivalent to the $\infty$-category of presheaves of rational spectra on the $\infty$-category whose objects are finite-dimensional Hilbert spaces, and whose hom-spaces are linear embeddings modulo scalars. We also point out some consequences that our models have for the $p$-localization and potentially chromatic localization of $\cM$ and $\NSp$.

\subsubsection*{Acknowledgements}
We would like to thank Jeffrey Carlson for fruitful correspondences during the early stages of our work. We are grateful to Vladimir Hinich for explaining to us his theory of enriched infinity categories and its relevance to our work.

\section{The functors $G_{k,l}$ and their stabilization}\label{ss: Glk}
In this section we recall the construction of $\cM$ as a category strictly enriched in $\SpM$ (see Remark \ref{r:M}). We will introduce certain functors $G_{k,l}\in\SpM$, which will represent the mapping spectra in $\cM$.


To begin with, let $\Csep$ denote the category of (non-unital) separable $C^*$-algebras and $*$-homomorphisms. Note that the matrix algebras $\{M_k\mid k=1, 2, \ldots\}$ are objects of $\Csep$. Consider $\Csep$ as a topologically enriched category, where for every $A,B\in\Csep$ we endow the set of $*$-homomorphisms $\Csep(A,B)$ with the topology of pointwise norm convergence. It is well-known that $\Csep$ is cotensored over the category of pointed finite CW-complexes~\cite{AG}. For a finite pointed CW-complex $X$ and a $C^*$-algebra $A$ we denote the contensoring by $\tC_0(X, A)$.

Next, we want to use $\Csep$ to define $\cM$ as a category enriched in $\SpM$. Recall that the underlying category of $\SpM$ is the category of pointed continuous functors from pointed finite CW-complexes to pointed topological spaces.


\begin{define}\label{def: Gkl}
Let $G_{k,l}\colon \FCW\to \Top$ be the functor defined as follows
\[
G_{k, l}(X)=\Map_{\Csep}(\tC_0(X, M_l),M_k).
\]
\end{define}
We will consider $G_{k, l}$ to be an object of $\SpM$.
Notice that for all $k, l, m$, there is a natural map $G_{k,l}(X)\wedge G_{l,m}(Y)\to G_{k, m}(X\wedge Y)$, defined as a composition of the following maps.
\begin{multline*}
\Map_{\Csep}(\tC_0(X, M_l),M_k)\wedge \Map_{\Csep}(\tC_0(Y, M_m),M_l)\to \\ \to\Map_{\Csep}(\tC_0(X, M_l),M_k)\wedge \Map_{\Csep}(\tC_0(X\wedge Y, M_m),\tC_0(X, M_l)) \to \\ \to \Map_{\Csep}(\tC_0(X\wedge Y, M_m), M_k).
\end{multline*}
Here the second map is composition, and the first map is induced by the cotensoring
\[
\Map_{\Csep}(\tC_0(Y, M_m),M_l)\to \Map_{\Csep}(\tC_0(X\wedge Y, M_m),\tC_0(X, M_l)).
\]
This map induces natural maps
\begin{equation}\label{eq: composition}
G_{k,l}\wedge G_{l,m} \to G_{k, m}
\end{equation}
 where $\wedge$ denotes internal smash product (aka Day convolution).
\begin{define}\label{def: cM}
Let $\cM$ be the following $\SpM$-enriched category. The objects of $\cM$ are positive integers. Given two integers $k, l$, the mapping spectrum from $k$ to $l$ is given by $G_{k, l}$. The composition law in $\cM$ is defined by the structure maps like in~\eqref{eq: composition}, for all $k, l, m$.
\end{define}

\begin{rem}
Recall that in the infinity categorical picture, $\cM$ is the full spectral subcategory of $\NSp$ whose objects are $\{\Sigma^\infty_{\NC} M_k\mid k=1, 2\ldots \}$. We show in \cite{ABS1} that the enriched coherent nerve of $\cM$ from Definition \ref{def: cM} is equivalent to $\cM$ defined above.
The objects of $\cM$ provide a set of compact generators of $\NSp$. The main result of~\cite{ABS1} says that the $\infty$-category $\NSp$ is equivalent to the $\infty$-category  $P_\Sp (\cM)$ of spectral presheaves on $\cM$. In this paper we investigate the category $\cM$, with the eventual goal in mind of understanding $\NSp$. In view of the results of~\cite{ABS1} in this paper we identify $\NSp$ with $P_\Sp (\cM)$. In the remaining part of the paper we will study the category $\cM$ mainly using the explicit model provided in Definition \ref{def: cM}, but will state the results also in the infinity categorical picture.
\end{rem}

Recall from the introduction that we identify $\SpM_\infty= \Sp$ and we denote by $\partial_1:\SpM\to \Sp$, the localization functor. If $G\in\SpM$ is a pointed continuous functors from pointed finite CW-complexes to pointed topological spaces, then $\partial_1 G$ is the spectrum corresponding to the sequence of spaces $\{G(S^0), G(S^1), \ldots $\}, or in other words
$$\partial_1 G\simeq
{\hocolim}_n\Sigma^{-n}\Sigma^{\infty}G(S^n).$$
This is known as the stabilization, or the first derivative of the functor $G$.
We have shown in \cite{ABS1} that for all natural numbers $k, l$, we have
\[
\SSS^{k,l}:=\Hom_{\NSp}(\Sigma^\infty_{\NC} M_k,\Sigma^\infty_{\NC} M_l) \simeq\partial_1G_{k, l}\simeq{\hocolim}_n\Sigma^{-n}\Sigma^{\infty}G_{k, l}(S^n).
\]



\begin{example}
Let us consider the case $k=l=1$. The functor $G_{1,1}$ is given as follows
\[
G_{1, 1}(X)=\Map_{\Csep}(\tC_0(X, M_1),M_1)=\Map_{\Csep}(\tC_0(X, \mathbb C),\mathbb C).
\]
By the Gelfand-Naimark theorem, it follows that $G_{1, 1}(X)\cong X$, and therefore $\SSS^{1,1}= \Sigma^\infty S^0$ is the ordinary sphere spectrum.
\end{example}
\begin{rem}\label{remark: commutative}
Let us interpret $\SSS^{1, 1}=\Sigma^\infty S^0$ as the endomorphism spectrum $\operatorname{End}_{\NSp}(\Sigma^\infty_{\NC} M_1)$. We have identified $\NSp$ with the $\infty$-category of spectral presheaves on $\cM$. In this picture, the $\infty$-category of spectral presheaves on the full $\Sp$-enriched subcategory of $\cM$ consisting of the object $\Sigma^\infty_{\NC} M_1$ can be identified with the $\infty$-category $\Sp$ of ``commutative'' or ``ordinary'' spectra. There is an ``inclusion'' functor of $\Sp$-tensored categories $\Sp\to \NSp$ which in terms of presheaves is defined by an $\Sp$-enriched left Kan extension (weighted colimit) from $\{\Sigma^\infty_{\NC} M_1\}$ to $\cM$. The inclusion functor has a right adjoint $\NSp\to \Sp$, a kind of ``abelianization'' functor, defined by restriction of presheaves. We will say a little more about it in Section~\ref{subsection: K-theory}.
\end{rem}

\section{The rank filtration}\label{section: stable rank}
In this section we introduce the rank filtration of $G_{k,l}$, which induces a rank filtration of the spectral category $\cM$. In later sections we will see that the rank filtration of $\cM$ is a lift of the classical rank filtration of the connective $K$-theory spectrum $ku$.

Let $l,k\geq 1$, let  $X$  be a finite pointed CW-complex and let
$$
f \in G_{k,l}(X)=\Csep(\tC_0(X, M_l),M_k),
$$
be a map. Let $A_f \subseteq M_k$ be the image of $f$. $A_f$ acts as non-unital $C^*$-Algebra on the Hilbert space  $\mathbb{C}^k$ and thus we get an orthogonal decomposition $\mathbb{C}^k = \mathrm{Ker}A_f \oplus A_f\cdot \mathbb C^k$. Denote $V_f := A_f\cdot \mathbb C^k \subseteq \mathbb{C}^k$.
We shall filter the space $G_{k,l}(X)$ according to the dimension of $V_f$. The following theorem is useful in that analysis.
\begin{thm}\label{t:ideal}
Let $X$ be a pointed compact metrizable space and let $l\geq 1$. There is a bijection between the closed subsets of $X\setminus\{*\}$ and closed two-sided ideals
of $\tC_0(X, M_l)$, defined by the following correspondence
$$F\mapsto I_F:=\{f\in\tC_0(X,M_l)\mid \forall x\in F .\, f(x)=0\}.$$
\end{thm}

\begin{proof}
The case $l=1$ is well-known. We will show that it implies the rest. Two $C^*$-algebras $A$ and $B$ are called strongly Morita equivalent if they are related by a $B$-$A$-imprimitivity bimodule in the sense of \cite{Rie}. Let $\KK$ be the algebra of compact operators on an infinite dimensional separable Hilbert space. It is shown in \cite{BGR} that if $A$ and $B$ are separable, then $A$ and $B$ are strongly Morita equivalent iff
$$A\otimes \KK\cong B\otimes \KK.$$

It is not hard to see that $\tC_0(X, M_l)\cong \tC_0(X)\otimes M_l$ and $M_l\otimes\KK\cong\KK$, so we have
$$\tC_0(X, M_l)\otimes\KK\cong (\tC_0(X)\otimes M_l)\otimes\KK \cong \tC_0(X)\otimes (M_l\otimes\KK) \cong \tC_0(X)\otimes \KK.$$
Thus, $\tC_0(X, M_l)$ and $\tC_0(X)$ are strongly Morita equivalent. By \cite{Zet}, we have an isomorphism between the sets of closed two-sided ideals
of $\tC_0(X, M_l)$ and of $\tC_0(X)$.
\end{proof}

\begin{lem}\label{lem: factorization}
Let $X \in \FCW$ be pointed finite CW-complex and let
$$f\in G_{k,l}(X)={\Csep}(\tC_0(X, M_l),M_k).$$
Then $f$ admits a unique factorization of the following form
\begin{equation}
\tC_0(X, M_l)  \twoheadrightarrow\tC_0(F_f\cup \{*\}, M_l)\stackrel{f'}{\rightarrowtail} M_k
\end{equation}
where the first map is the surjective restriction to a finite subset $F_f \subset X\setminus \{\ast\}$ and the second map $f'\colon \tC_0(F_f\cup \{*\}, M_l) \rightarrowtail M_k$ is a monomorphism. In particular we have $V_f = V_{f'}$.
\end{lem}
\begin{proof}
First, note that $\ker(f)$ is a closed two sided $*$-ideal of $\tC_0(X, M_l)$. By Theorem \ref{t:ideal}, there exists a closed subset $F_f$ of $X\setminus\{*\}$ such that
$$\ker(f)=I_{F_f}=\{g\in \tC_0(X, M_l)\mid g|_{F_f}=0\}.$$
Notice that $\ker(f)=I_{F_f}$ is also the kernel of the restriction homomorphism
$$\tC_0(X, M_l)\to \tC_0(F\cup\{*\}, M_l).$$
We claim that the restriction homomorphism is surjective. This amounts to showing that any map from a closed subset of $X$ to $M_l$ can be extended to a map from $X$ to $M_l$. This in turn follows immediately from the Tietze extension theorem.

It follows that $f$ admits a unique factorization of the following form
\begin{equation}\label{eq: factorization}
\tC_0(X, M_l) \twoheadrightarrow \tC_0(F_f\cup \{*\}, M_l)\stackrel{f'}{\rightarrowtail} M_k
\end{equation}
where the first map is restriction to a subset $F_f$ and the second map $f'\colon \tC_0(F_f\cup \{*\}, M_l) \to M_k$ is a monomorphism. Moreover $F_f$ is the minimal subset of $X\setminus \{*\}$ for which the map $f$ factors through $ \tC_0(F_f\cup \{*\}, M_l)$. Notice that $\im(f')$ is finite-dimensional as a vector space over $\CC$. This implies that $F_f$ is finite.
\end{proof}

\begin{lem}\label{l:lele}
let
$$
f \in G_{k,l}(X)=\Csep(\tC_0(X, M_l),M_k),
$$
we have $$l|\dim V_f$$ and $$l\cdot |F_f | \leq  \dim V_f \leq k$$
\end{lem}
\begin{proof}
By lemma \ref{lem: factorization}   the function $f$ can be factored as a surjection followed by an injection
\[
\tC_0(X, M_l) \to M_l^{F_f} \stackrel{f'}{\to} M_k
\]
Thus we have an isomorphism $M_l^{F_f} \cong A_f$. $1\in M_l^{F_f}$ now acts on $\mathbb{C}^k$  as a projection onto $V_f$ and thus we
get a \textbf{unital} action of $M_l^{F_f}$ on $V_f$.  Now for $ x \in  F_f$ denote by $W_x$ the unital $M_l^{F_f}$ module obtained by the canonical action on $\mathbb{C}^l$ via map $M_l^{F_f}\to M_l^{\{x\}}  = M_l$. Every finite dimensional unital
 $M_l^{F_f}$ module is a direct sum of finitely many copies of the $W_x$'s. We thus get that $$V_f  = \bigoplus_{x \in F_f} = W_x^{e_x}.$$  The injectivity   of the map $M_l^{F_f} \stackrel{f'}{\to} M_k$ implies that  $e_x \ge 1$ for every $x \in F_f $. Since $\dim V_f = l\sum e_x$ and $V_f \subseteq \mathbb{C}^k$ we get the claim.
\end{proof}

\begin{define}
Let $l,k\geq 1$, let  $X$  be a finite pointed CW-complex and let
$$
f \in G_{k,l}(X)=\Csep(\tC_0(X, M_l),M_k).
$$
We define the \emph{rank of $f$} to be the non-negative integer
  $$\mathrm{rank}(f):= \frac{\dim V_f}{l} \in \mathbb{Z}_{\ge 0}.$$
\end{define}
Suppose we have a map $\alpha\colon X\to Y$ in $\FCW$. By functoriality, it induces a map $G_{k,l}(X)\to G_{k,l}(Y)$. Suppose $f\in G_{k,l}(X)$. By definition, $f$ is a $*$-homomorphism $f\colon \tC_0(X, M_l) \to M_k$, and the image of $f$ in $G_{k,l}(Y)$ is the composite homomorphism
\[
\tC_0(Y, M_l)\xrightarrow{\alpha^*} \tC_0(X, M_l) \xrightarrow{f} M_k.
\]
Therefore the rank of the image of $f$ in $G_{k,l}(Y)$ is at most the rank of $f$. Because of this, the following definition really does describe a functor.
\begin{define}  Let $k, l \ge 1$, $ m\ge 0$. Define the functors $G_{k,l,m} \colon \FCW  \to \Top$ as follows
\[
G_{k,l,m}(X)= \left \{ f \in G_{k,l}(X) |\quad\mathrm{rank}(f) \leq m \right \} \subseteq G_{k,l}(X).
\]
Similarly, define $\SSS^{k,l,m}$ to be the stabilization of $G_{k,l,m}$. Explicitly, $\SSS^{k,l,m}$ is the spectrum $\{G_{k,l,m}(S^0), G_{k,l,m}(S^1), \ldots\}$.
\end{define}
\begin{rem}
Note that for all $X \in \FCW$ and $k, l \ge 1$  we have $G_{k,l,0}(X) = \ast $. Additionally by lemma \ref{l:lele}
for $m \geq \lfloor\frac{k}{l}\rfloor $ we get $G_{k,l,m}  = G_{k,l}$.
\end{rem}
We have defined a filtration of $G_{k,l}$ by sequence of subfunctors
\[
*=G_{k,l,0}\subset G_{k,l,1}\subset \cdots \subset G_{k,l, \lfloor\frac{k}{l}\rfloor} = G_{k,l}.
\]
We call this filtration \emph{the rank filtration}. Now recall that the functors $G_{k,l}$ represent mapping spectra in $\cM$ and that composition in $\cM$ is determined by maps of the following form
\[
G_{k,l}(X)\wedge G_{l,m}(Y) \to G_{k,m}(X\wedge Y).
\]
The following proposition tells how the rank filtration interacts with composition.
\begin{prop}\label{prop: rank}
For all $r$ and $s$, the composition map above restricts to a natural map
\[
G_{k,l,r}(X)\wedge G_{l,m,s}(Y) \to G_{k,m,rs}(X\wedge Y)
\]
\end{prop}
\begin{proof}
Recall that $G_{k,l}(X)= \Csep(\tC_0(X, M_l),M_k)$. Written in these terms, the composition map has the following form
\begin{multline*}
\Csep(\tC_0(X, M_l),M_k)\wedge \Csep(\tC_0(Y, M_m),M_l)\to \\ \to \Csep(\tC_0(X, M_l),M_k)\wedge  \Csep(\tC_0(X\wedge Y, M_m),\tC_0(X, M_l))\to \\ \to \Csep(\tC_0(X\wedge Y, M_m), M_k).
\end{multline*}
Suppose that $f\in \Csep(\tC_0(X, M_l),M_k)$ and $g\in \Csep(\tC_0(Y, M_m),M_l)$ have ranks $r$ and $s$ respectively. Let $f\odot g$ denote the image of $f\wedge g$ in $\Csep(\tC_0(X\wedge Y, M_m), M_k)$. Our goal is to show that $f\odot g$ has rank $rs$.

By Lemma~\ref{lem: factorization}, there exist finite subsets $F_f\subset X\smallsetminus \{\ast\}$, $F_g\subset Y\smallsetminus \{\ast\}$ such that $f$ factors as $\tC_0(X, M_l)  \twoheadrightarrow\tC_0(F_f\cup \{*\}, M_l)\stackrel{f'}{\rightarrowtail} M_k$, and there is a similar factorization of $g$. It follows that $f\odot g$ factors as follows
\[
\tC_0(X\wedge Y, M_m)  \twoheadrightarrow\tC_0(F_f\times F_g \cup \{*\}, M_m)\stackrel{f'\odot g'}{\rightarrowtail} M_k
\]
where the second map is itself the following composite
\[
(M_m^{F_g})^{F_f}\xrightarrow{g'^{\times F_f}} M_l^{F_f} \xrightarrow{f'} M_k.
\]
Here the first map is the cartesian product of $|F_f|$ copies of the $g'$ with itself. This map determines an action of $M_m^{F_g\times F_f}$ on $\mathbb C^k$. Our goal is to show that $M_m^{F_g\times F_f}\cdot \mathbb C^k$ has dimension $rsm$.

Recall that $A_f=A_{f'}$ is the image of $f'$. Since $\mathrm{rank}(f)=r$, $A_f\cdot \mathbb C^k$ has dimension $rl$. If $B\subset M_l$ is a $C^*$-subalgebra such that $B\cdot \mathbb C^l$ has dimension $d$, and we let $B^{F_f}$ act on $\mathbb C^k$ via the map $f'$, then $B^{F_f}\cdot \mathbb C^k$ has dimension $rd$. Now take $B$ to be the image of $g'$.
Since $\mathrm{rank}(g)=s$, $B\cdot \mathbb C^l$ has dimension $sm$, so finally we conclude that $M_m^{F_g\times F_f}\cdot \mathbb C^k$ has dimension $rsm$. This means that $f\odot g$ has rank $rs$.
\end{proof}
\begin{rem}
Proposition~\ref{prop: rank} can be intepreted as follows: the rank filtration is a filtration of the category $\cM$ by the {\it multiplicative} monoid of natural numbers.
\end{rem}
\section{The restriction of $G_{k,l}$ to finite sets}\label{ss: finite}
It will turn out that the functor $G_{k,l}$, whose domain is the category of pointed finite CW-complexes, is determined by its restriction to the category of pointed finite sets. In this section we give an explicit description of the restriction of $G_{k,l}$ to finite sets.

Let us begin with a definition, which also serves to establish some notation.
\begin{define}
For a natural number $i$, let $[i]=\{0,1,\ldots,i\}$, considered as a pointed set with basepoint $0$. Let $\Fin$ be the category whose objects are $\{[0],[1],\dots,[k], \ldots\}$ and whose morphisms are basepoint-preserving functions. For $k\geq 0$, let $\Finsk$ denote the full subcategory of $\Fin$ spanned by the objects $\{[0],[1],\dots,[k]\}$. We will also use the notation $\underline{i}$ for the unpointed set $\{1, \ldots, i\}$.
\end{define}
\begin{rem}
The category $\Fin$ is denoted $\Gamma$ in some sources, and $\Gamma^{\op}$ in some other sources. We find the notation $\Fin$ to be more descriptive. But following the tradition established by Segal~\cite{Se1}, we call pointed functors $\Fin\to \Top$ {\it $\Gamma$-spaces}.
\end{rem}
We will now examine the restriction of $G_{k,l}$ to $\Fin$. For a finite pointed set $[t]$, $G_{k,l}([t])$ is the space of non-unital $C^*$-algebra homomorphisms from $M_l^t$ to $M_k$. Spaces of such homomorphisms are well-understood. We want to describe them in a way that makes the functoriality in $[t]$ explicit.
%
%
%
We need a few definitions.
\begin{define}\label{def: multisets}
The category of pointed multisets is defined as follows. The objects are ordered $t$-tuples $(m_1, \ldots, m_t)$ of natural numbers. The possibility $t=0$ is included, in which case the tuple is empty. A morphism $(m_1, \ldots, m_t)\to (n_1, \ldots, n_s)$ consists of a pointed function $\alpha\colon [t]\to [s]$ such that $n_j=\Sigma_{i\in \alpha^{-1}(j)}m_i$ for all $1\le j\le s$. In particular, if $j$ is not in the image of $\alpha$ then $n_j=0$. Note that there are no restrictions on $m_i$ for $i\in \alpha^{-1}(0)$.

Given a pointed multiset $(m_1, \ldots, m_t)$ and a pointed function of sets $\alpha\colon [t]\to [s]$ we define $\alpha_*(m_1, \ldots, m_t)$ to be the multiset $(n_1, \ldots, n_s)$ with $n_j=\Sigma_{i\in \alpha^{-1}(j)}m_i$ for all $1\le j\le s$.
\end{define}
\begin{example}
Let $\alpha\colon[3]\to [2]$ be the function defined by $\alpha(0)=\alpha(1)=0$, $\alpha(2)=\alpha(3)=1$. Then $\alpha_*(4,2,3)=(5,0)$.
\end{example}
Suppose we have a multiset $(m_1, \ldots, m_t)$ and a natural number $l$. We will make much use of the unitary vector space $\CC^{(m_1+\cdots+m_t)l}$. We identify this vector space with
\[
\CC^{m_1+\cdots+m_t}\otimes \CC^l\cong \CC^{m_1}\otimes \CC^l \oplus\cdots \oplus \CC^{m_t}\otimes \CC^l.
\]
Notice that there are commuting actions of $U(m_1)\times \cdots \times U(m_t)$ and $U(l)$ on $\CC^{(m_1+\cdots+m_t)l}$. It follows that these groups act on any space obtained by applying a continuous functor to this vector space.

Now suppose we have a morphism of pointed multisets $\alpha\colon (m_1, \ldots, m_t)\to (n_1, \ldots, n_s)$, so $(n_1, \ldots, n_s)=\alpha_*(m_1, \ldots, m_t)$. Choose unitary isomorphisms $\CC^{n_j}\stackrel{\cong}{\to} \CC^{\Sigma_{i\in\alpha^{-1}(j)} m_i}$ for all $1\le j\le s$. The function $\alpha$ together with these isomorphisms determine an inner-product-preserving inclusion $\CC^{(n_1+\cdots+n_s)l}\to \CC^{(m_1+\cdots+m_t)l}$.
This inclusion in turn defines a map of spaces (where $k$ is another natural number)
\[
\Inj(\CC^{(m_1+\cdots+m_t)l}, \CC^k)\to \Inj(\CC^{(n_1+\cdots+n_s)l}, \CC^k)
\]
A different choice of isomorphisms $\CC^{n_j}\stackrel{\cong}{\to} \CC^{\Sigma_{i\in\alpha^{-1}(j)} m_i}$ will change the map by precomposition with a unitary automorphism of $\CC^{(n_1+\cdots+n_s)l}$ that is induced by automorphisms of $\CC^{n_1}, \ldots, \CC^{n_s}$. Therefore we get a well-defined (i.e., independent of choices of isomorphisms) map
\[
\Inj(\CC^{(m_1+\cdots+m_t)l}, \CC^k)\to \Inj(\CC^{(n_1+\cdots+n_s)l}, \CC^k)/_{\prod_{j=1}^s U(n_j)}
\]
Moreover, it is easy to see that the map passes to a well-defined map between quotients
\begin{equation}\label{eq: functor}
\Inj(\CC^{(m_1+\cdots+m_t)l}, \CC^k)/_{\prod_{i=1}^t U(m_i)} \to \Inj(\CC^{(n_1+\cdots+n_s)l}, \CC^k)/_{\prod_{j=1}^s U(n_j)}
\end{equation}
The upshot is that we have defined a functor from the category of pointed multi-sets to spaces that sends $(m_1, \ldots, m_t)$ to $\Inj(\CC^{(m_1+\cdots+m_t)l}, \CC^k)/_{\prod_{i=1}^t U(m_i)} $.
\begin{rem}
Here is a slightly different way to think of the map~\eqref{eq: functor}. Let $m=m_1+\cdots+m_t$. There are homeomorphisms
\[
\Inj(\CC^{ml}, \CC^k)/_{\prod_{i=1}^t U(m_i)}\cong U(k)/\prod_{i=1}^t U(m_i) \times U(k-ml)
\]
and similarly
\[
\Inj(\CC^{nl}, \CC^k)/_{\prod_{j=1}^s U(n_j)}\cong U(k)/\prod_{j=1}^s U(n_j) \times U(k-nl).
\]
A morphism of multisets $\alpha\colon (m_1, \ldots, m_t)\to (n_1, \ldots, n_s)$ gives a canonical way to conjugate $\prod_{i=1}^t U(m_i) \times U(k-ml)$ into a subgroup of $\prod_{j=1}^s U(n_j) \times U(k-nl)$, and therefore gives rise to a $U(k)$-equivariant map
\[
U(k)/\prod_{i=1}^t U(m_i) \times U(k-ml)\to U(k)/\prod_{j=1}^s U(n_j) \times U(k-nl).
\]
\end{rem}

Now we can describe the functor $G_{k,l}$ on finite sets. The following proposition is essentially due to Bratelli~\cite{Bratelli}.
\begin{prop}\label{prop:SC inj Mlt Mk}
There is a homeomorphism
\begin{equation}\label{eq: Gkloft}
G_{k,l}([t])\cong \bigvee_{(m_1, \ldots, m_t)} {\Inj(\CC^{(m_1+\cdots+m_t)l}, \CC^k)/_{\prod_{i=1}^t U(m_i)}}_+
\end{equation}
The wedge sum on the right is indexed on non-zero ordered $t$-tuples $(m_1, \ldots, m_t)$ of non-negative integers (the zero tuple corresponds to the basepoint). The functoriality on the right hand side is defined as follows. A pointed map $\alpha\colon [t]\to [s]$,
induces a map
\begin{multline*}
\bigvee_{(m_1, \ldots, m_t)} {\Inj(\CC^{(m_1+\cdots+m_t)l}, \CC^k)/_{\prod_{i=1}^t U(m_i)}}_+ \to \\ \to\bigvee_{(n_1, \ldots, n_s)} {\Inj(\CC^{(n_1+\cdots+n_s)l}, \CC^k)/_{\prod_{j=1}^s U(n_j)}}_+
\end{multline*}
that sends the wedge summand corresponding to $(m_1, \ldots, m_t)$ to the wedge summand corresponding to $\alpha_*(m_1, \ldots, m_t)$ by the map~\eqref{eq: functor}, assuming $\alpha_*(m_1, \ldots, m_t)$ is not a tuple of zeros. If $\alpha_*(m_1, \ldots, m_t)$ consists just of zeros, then $\alpha$ sends the corresponding wedge summand to the basepoint.
\end{prop}
\begin{proof}
By definiton~\ref{def: Gkl}, there is a homeomorphism \[G_{k,l}([t])\cong  {\Csep}(M_l^t,M_k).\]
For every multi-set $(m_1, \ldots, m_t)$, we define a map
\begin{equation}\label{eq: basicmap}
\Inj(\CC^{(m_1+\cdots+m_t)l}, \CC^k)/_{\prod_{i=1}^t U(m_i)} \to {\Csep}(M_l^t,M_k)
\end{equation}
as follows. Suppose we have a unitary isometric inclusion $\CC^{(m_1+\cdots+m_t)l} \hookrightarrow \CC^k$. From this, we get a unitary isomorphism (determined up to an automorphism of $\CC^{k-ml}$)
\begin{equation}\label{eq: decomp}
 \CC^{m_1}\otimes \CC^l \oplus \cdots \oplus\CC^{m_t}\otimes \CC^l \oplus \CC^{k-ml} \stackrel{\cong}{\to} \CC^k
\end{equation}
Having fixed such an isomorphism, we associate with it a $C^*$-algebra homomorphism $M_l^t\to M_k$ as follows: the $i$-th factor $M_l$ of $M_l^t$ acts on $\CC^{m_i}\otimes \CC^l$ by identity on $\CC^{m_i}$ and by the standard action on $\CC^l$. Note that the action of $M_l^t$ on $\CC^{k-ml}$ is multiplication by zero.

Automorphisms of $\CC^{k-ml}$ commute with the action of $M_l^t$ on $ \CC^{m_1}\otimes \CC^l \oplus \cdots \oplus\CC^{m_t}\otimes \CC^l \oplus \CC^{k-ml}$. It follows that changing isomorphism~\ref{eq: decomp} by an automorphism of $\CC^{k-ml}$ does not change the resulting algebra homomorphism from $M_l^t$ to $M_k$. It follows in turn that we have a well-defined map
\[
\Inj(\CC^{(m_1+\cdots+m_t)l}, \CC^k) \to {\Csep}(M_l^t,M_k).
\]
It follows from elementary representation theory (Schur Lemma) that two elements of $\Inj(\CC^{(m_1+\cdots+m_t)l}, \CC^k)$ induce the same algebra homomorphism if and only if they differ by an action of $U(m_1)\times\cdots\times U(m_t)$. Therefore we get a well-defined injective map in~\eqref{eq: basicmap}.

Taking union over multi-sets of the form $(m_1, \ldots, m_t)$ with fixed $t$, we obtain a map
\begin{equation}\label{eq: algebra morphisms}
 \bigvee_{(m_1, \ldots, m_t)} {\Inj(\CC^{(m_1+\cdots+m_t)l}, \CC^k)/_{\prod_{i=1}^t U(m_i)}}_+ \stackrel{\cong}{\to} {\Csep}(M_l^t,M_k)
\end{equation}
which we claim is a homeomorphism. Indeed, we already know that it is injective. 
Next we need to show that the map~\eqref{eq: algebra morphisms} is surjective. Let $f\colon M_l^t\to M^k$ be a $C^*$-algebra homomorphism. Let $I_1, \ldots, I_t$ be the identity elements of the $t$ factors $M_l$ of $M_l^t$. Then $f(I_1), \ldots, f(I_t)$ are pairwise commuting hermitian idempotents in $M_k$. It follows that for $i=1, \ldots, t$, $f(I_i)$ is hermitian projection onto $U_i$, where $U_1, \ldots, U_t$ are pairwise orthogonal subspaces of $\CC^k$. Now suppose that $1 \le i\le t$ and $A_i$ is an element of the $i$th factor $M_l$ of $M_l^t$. Then $f(A_i)=f(A_i)f(I_i)=f(I_i)f(A_i)$. Thus $f(A_i)$ commutes with the hermitian idempotent $f(I_i)$. It follows that $f(A_i)$ leaves invariant $U_i$ and the orthogonal complement of $U_i$. Moreover, since $f(A_i)=f(A_i)f(I_i)$ it follows that $f(A_i)$ is the composition of projection onto $U_i$ and a linear transformation of $U_i$. It follows that the restriction of $f$ to the $i$th factor of $M_l^t$ defines a unital representation of the algebra $M_l$ on $U_i$. Since $M_l$ is Morita equivalent to $\CC$, $U_i$ is isomorphic to a sum of copies of the standard representation of $M_l$. This means that we can write $U_i\cong \CC^{m_i}\otimes \CC^l$, where $m_1, \ldots, m_t$ are some non-negative integers. With this identification the $i$-th $M_l$ acts on $U_i$ via standard action, and it follows that $f$ is in the image of the map~\eqref{eq: algebra morphisms}

We have shown that the map~\eqref{eq: algebra morphisms} is a bijection. To show that it is a homeomorphism, observe that $U(k)$ acts continuously on the source and the target. Moreover, both the source and the target are topologized as the disjoint union of $U(k)$-orbits. This is true by definition for the source. To see this for the target, notice that the map that associates to an algebra morphism $f\colon M_l^t\to M_k$ the integers $(m_1, \ldots, m_t)$ is continuous and therefore locally constant, and $U(k)$ acts transitively on the preimage of any $t$-tuple of integers. Thus the map ~\eqref{eq: algebra morphisms} is a $U(k)$-equivariant bijection between disjoint unions of orbits of a continuous action $U(k)$. It follows that it is a homeomorphism.

The statement about functoriality follows by straightforward diagram-chasing.
\end{proof}
%
In the previous section we defined the rank filtration of $G_{k,l}$. Unwinding the definitions, we find that if $f \in G_{k,l}([t])$ belongs to the wedge summand corresponding to $(m_1, \ldots, m_t)$ in Proposition \ref{prop:SC inj Mlt Mk}, then
$$\mathrm{rank}(f)=m_1+\cdots + m_t.$$
Thus, on finite sets the rank filtration is given by the following formula:
\begin{equation}\label{eq: rank filtration}
G_{k,l,m}([t])= \bigvee_{\underset{ m_1+\cdots + m_t\le m\}}{\{(m_1, \ldots, m_t)\mid}} {\Inj(\CC^{(m_1+\cdots+m_t)l}, \CC^k)/_{\prod_{i=1}^t U(m_i)}}_+.
\end{equation}

\section{$G_{k,l}$ is a cofibrant $\Gamma$-space}\label{section: cofibrant}
In this section we observe that for all $k, l, m$ the restriction of the functor $G_{k,l,m}$ from finite complexes to finite sets is cofibrant, in a certain well-known model structure on $\Gamma$ spaces. This implies that the strict smash product between these functors is equivalent to the derived smash product. We also show that the value of the functor $G_{k,l,m}$ on pointed finite CW-complexes is equivalent to both the strict and the derived left Kan extension of the restriction of $G_{k,l,m}$ to the category of pointed finite sets. Furthermore the $\Gamma$-space $G_{k,l,m}$ is $\min(\lfloor\frac{k}{l}\rfloor, m)$-skeletal. This implies that $G_{k,l,m}$ is determined by its restriction to the category of sets of cardinality at most $\min(\lfloor\frac{k}{l}\rfloor, m)$.

For any fixed $[t]$, there is a canonical map \begin{equation}\label{eq: skeletal inclusion} \colim_U G_{k,l,m}(U_+) \to G_{k,l,m}([t])\end{equation} where $U$ ranges over the poset of {\em proper} subsets of $\underline{t}=\{1, \ldots, t\}$ and $U_+=U\cup \{0\}$. Recall once again that there is an isomorphism

\begin{equation}
G_{k,l,m}([t])= \bigvee_{\underset{ m_1+\cdots + m_t\le m\}}{\{(m_1, \ldots, m_t)\mid}} {\Inj(\CC^{(m_1+\cdots+m_t)l}, \CC^k)/_{\prod_{i=1}^t U(m_i)}}_+.
\end{equation}

With this isomorphism in mind, the following lemma is proved by routine manipulations of colimits
\begin{lem}
There is an isomorphism
\[\colim_{U\subsetneq \underline{t}} G_{k,l,m}(U_+)\cong
\bigvee_{\substack{ \{(m_1, \ldots, m_t)\mid m_i=0 \mbox{ \small for some }1\le i\le t \\  \mbox{ and } m_1+\cdots + m_t\le m\}}} {\Inj(\CC^{(m_1+\cdots+m_t)}, \CC^k)/_{\prod_{i=1}^t U(m_i)}}_+
\]
The map
\[
\colim_{U\subsetneq \underline{t}} G_{k,l}(U_+)\to G_{k,l}([t])
\]
corresponds, under this isomorphism, to inclusion of the wedge sum of all summands indexed by tuples $(m_1, \ldots, m_t)$ where $m_i=0$ for at least one $i$.
\end{lem}
It follows that the map~\eqref{eq: skeletal inclusion} is an inclusion of a union of path components. Furthermore, the action of $\Sigma_t$ on the quotient space of this inclusion is free. This means that~\eqref{eq: skeletal inclusion} is a $\Sigma_t$-equivariant cofibration, and this in turn means that as a functor on $\Fin$, $G_{k,l,m}$ is cofibrant in the model structure defined for $\Gamma$-spaces in~\cite{BF, Lyd1} (technically, there references work with $\Gamma$-simplicial sets, but an analogous structure exists for $\Gamma$-spaces). This model structure is also discussed in~\cite{BM} as an example of a generalized Reedy model structure. We will refer to this model structure simply as the Reedy model structure.

Let us note that the quotient space of~\eqref{eq: skeletal inclusion} is given by the wedge sum
\[
\bigvee_{\substack{ \{(m_1, \ldots, m_t)\mid m_1,\ldots, m_t \ge 1 \\  \mbox{ and } m_1+\cdots + m_t\le m\}}} {\Inj(\CC^{(m_1+\cdots+m_t)l}, \CC^k)/_{\prod_{i=1}^t U(m_i)}}_+.
\]
Notice that the quotient is trivial for $t>m$ and for $t>\lfloor\frac{k}{l}\rfloor$.  In the terminology of~\cite{BF, Lyd1}, this means that the $\Gamma$-space $G_{k,l,m}$ is $\min(m, \lfloor\frac{k}{l}\rfloor)$-skeletal.
This implies that $G_{k,l,m}$ is determined on $\Fin$, via left Kan extension, by its restriction to the subcategory of sets of cardinality at most $\min(m, \lfloor\frac{k}{l}\rfloor)$.

We want to verify that $G_{k,l}$ is equivalent, as a functor on $\FCW$, to both the strict and the derived left Kan extension of the restriction of $G_{k,l,m}$ to $\Fin$, and even to $\Fin^{\le \min(m, \lfloor\frac{k}{l}\rfloor)}$, where $\Fin^{\le j}\subset \Fin$ is the full subcategory consisting of $[0], \ldots, [j]$. Left Kan extension can be described as a coend. Let us introduce some notation.

Let $\Top$ denote the category of pointed compactly generated weak Hausdorff spaces with the standard model structure of Quillen \cite{Qui}. Every object in $\Top$ is fibrant, and every CW-complex is cofibrant. Suppose $\cC$ is a small category, and we have a pair of functors $F\colon \cC^{\op}\to\Top$ and $G\colon \cC\to \Top$. We denote the coend of $F$ and $G$ by $\int_s^{\cC} F\wedge G$, or $\int_s^{x\in \cC}F(x)\wedge G(x)$. The subscript $s$ is there to indicate that this is a {\it strict} coend, as opposed to the {\it homotopy} coend. Let us recall the definition and construction of the latter.


\begin{define}
For a functor $F$, let $Q_oF$ denote an objectwise cofibrant replacement of $F$ and let $Q_pF$ denote a cofibrant replacement in a projective model structure. If $\cC$ is a generalized Reedy category in the sense of~\cite{BM}, then let $Q_rF$ denote a cofibrant replacement in the Reedy model structure.
\end{define}
The following lemma is standard
\begin{lem}\label{lemma: coends}
There are natural equivalences
\[
\int_s^{\cC} Q_pF\wedge Q_pG \simeq \int_s^{\cC} Q_oF\wedge Q_pG \simeq \int_s^{\cC} Q_pF\wedge Q_oG \simeq \int_s^{\cC} Q_rF\wedge Q_rG
\]
\end{lem}
\begin{proof}
Let us prove, for example, the equivalence $\int_s^{\cC} Q_pF\wedge Q_oG \simeq \int_s^{\cC} Q_rF\wedge Q_rG$. The other equivalences are similar. It is enough to prove that for all spaces $Z$ there is an equivalence, natural in $Z$
\[
\Map_{\Top}\left(\int_s^{\cC} Q_pF\wedge Q_oG, Z\right) \simeq \Map_{\Top}\left(\int_s^{\cC} Q_rF\wedge Q_rG, Z\right).
\]
Note that these are derived mapping spaces as the source is cofibrant and the target fibrant. For any covariant/contravariant pair of functors $F, G$, there is a homeomorphism
\[
\Map_{\Top}\left(\int_s^{\cC} F\wedge G, Z\right) \cong \Nat(F, \Map(G, Z)).
\]
Therefore it is enough to show that there is a natural equivalence
\begin{equation}\label{eq: nats}
\Nat\left(Q_pF, \Map( Q_oG, Z)\right) \simeq \Nat\left(Q_rF, \Map(Q_rG, Z)\right).
\end{equation}
The key observation is that if $Q_rG(-)$ is Reedy cofibrant then the functor $\Map(Q_rG(-), Z)$ is Reedy fibrant. The two functors $\Map( Q_oG(-), Z)$ and $\Map(Q_rG(-), Z)$ are weakly equivalent functors from $\cC^{\op}$ to $\Top$. They are fibrant in the projective and the Reedy model structure respectively. On the other hand, the functors $Q_pF$ and $Q_rF$ are weakly equivalent functors that are cofibrant in the projective and the Reedy model structure. It follows that the two sides of~\eqref{eq: nats} are the derived mapping spaces from $F$ to $\Map(Q_o(G), Z)$ in the projective and the Reedy model structure respectively. Since the two model structures have the same week equivalences, the two derived mapping spaces are equivalent.
\end{proof}
\begin{define}
The homotopy coend of $F$ and $G$, denoted by $\int_h^{\cC} F\wedge G$, is defined to be any one of the equivalent coends in Lemma~\ref{lemma: coends}.
\end{define}

\begin{rem}
As mentioned in the introduction, we identify the $\infty$-localization of $\Top$ with the $\infty$-category of pointed spaces $\Top_\infty=\cS_*.$
Let $F_\infty\colon \cC^{\op}\to\cS_*$ and $G_\infty\colon \cC\to \cS_*$ be the compositions of $F$ and $G$ with the localization functor $\Top\to\Top_\infty$. Then it is known (see for example \cite[Proposition 2.5.6]{BHH}) that the image of the homotopy coend
of $F$ and $G$ under the localization functor $\Top\to\Top_\infty$ is equivalent to the $\infty$-coend of $F_\infty$ and $G_\infty$.
In the sequel we will sometimes abuse notation and identify the homotopy coend with the $\infty$-coend and more generally homotopy colimits with $\infty$-colimits.
\end{rem}

The case we are interested in is of the covariant functor $G_{k,l,m}\colon \Fin\to \Top$ and the contravariant functor $X^{-}\colon \Fin^{\op}\to \Top$, where $X$ is a pointed finite CW complex. The strict coend $\int^{[t]\in \Fin}_{\mathrm s}X^t\wedge G_{k,l,m}([t])$ is a model for the strict (continuous) Kan extension of $G_{k,l,m}$ from $\Fin$ to $\Top$. The homotopy coend
of same functors is a model for the homotopy Kan extension. 
We saw above that $G_{k,l,m}$ is cofibrant in the Reedy model structure. The functor $X^{-}$ is also Reedy cofibrant if $X$ is a CW complex. This amounts to saying that for all $t$, the inclusion of the fat diagonal $\Delta^tX$ into $X^t$ is a $\Sigma_t$-equivariant cofibration. Since both functors are Reedy cofibrant, their homotopy coend is in fact equivalent to the strict coend.

\begin{lem}\label{lemma: strict is homotopy}
Let $X$ be a pointed CW complex and $j$ any integer satisfying $\infty \ge j\ge \min(m, \lfloor\frac{k}{l}\rfloor)$. All the maps in the following diagram are equivalences
\[
\begin{array}{ccc}
\int^{[t]\in \Fin^{\le j}}_{\mathrm s}X^t\wedge G_{k,l,m}([t]) & \to & \int^{[t]\in \Fin}_{\mathrm s}X^t\wedge G_{k,l,m}([t]) \\
\downarrow & &\downarrow \\
\int^{[t]\in  \Fin^{\le j}}_{\mathrm h}X^t\wedge G_{k,l,m}([t]) & \to & \int^{[t]\in \Fin}_{\mathrm h}X^t\wedge G_{k,l,m}([t])
\end{array}
\]
\end{lem}
\begin{proof}
The vertical maps are equivalences because the functors $X^-$ and $G_{k,l,m}$ are each Reedy cofibrant. The top map is an equivalence because $G_{k,l,m}$ is $j$-skeletal.
\end{proof}
Now comes the main result of this section: $G_{k,l,m}$ is equivalent to both the strict and the derived left Kan extension of its restriction to the category of finite sets of size at most $\min(m, \lfloor\frac{k}{l}\rfloor)$.
\begin{thm}\label{t:polynom inverse}
For all $k, l, m,$ and $\infty\ge j\ge \min(m, \lfloor\frac{k}{l}\rfloor)$, the functor $G_{k,l,m}$ is equivalent to both the strict and the derived left Kan extension of $G_{k,l,m}|_{\Fin^{\le j}}$ along the inclusion $\Fin^{\le j}\subseteq \CW$.
\end{thm}
\begin{proof}
By lemma~\ref{lemma: strict is homotopy}, the map from the derived left Kan extension to the strict one is an equivalence. So it is enough to prove the statement for strict Kan extension. So throughout this proof, $\int$ stands for the strict coend. There is a natural assembly map
\begin{equation}\label{eq: assembly}
\int^{[t]\in \Fin} X^t \wedge G_{k,l,m}([t])\to G_{k,l,m}(X).
\end{equation}
We will prove that it is a homeomorphism, if $X$ is a finite CW complex. This is enough for proving the theorem. Thus we need to prove that~\eqref{eq: assembly} is bijective and bi-continuous.

First we prove surjectivity. Let
$$f\in G_{k,l,m}(X)\subseteq{\Csep}(\tC_0(X, M_l),M_k).$$ Factor $f$ using lemma \ref{lem: factorization}.
Denote $t = |F_f|$ and choose a pointed bijection $[t]\xrightarrow{\cong} F_f\cup \{*\}$. It follows that $f$ admits a factorization
\[
\tC_0(X, M_l) \to M_l^t\stackrel{f'}{\to} M_k.
\]
Note that since $V_f = V_f'$ we have $f' \in G_{k,l,m}([t])$
This means that $f$ is in the image of the map $X^t\wedge G_{k,l,m}([t])\to G_{k,l,m}(X)$. Thus $f$ is in the image of the assembly map. Since $f$ was an arbitrary element of $G_{k,l,m}(X)$, we have proved surjectivity of the assembly map.

Next we show that the assembly map is injective. Suppose that $(\alpha, g)\in X^{t} \wedge G_{k,l,m}([t])$ and $(\alpha_1, g_1)\in X^{t_1} \wedge G_{k,l,m}([t_1])$ represent two elements of $\int^{[t]\in \Fin} X^t \wedge G_{k,l}([t])$ that are mapped to the same element of $G_{k,l,m}(X)$. We have to show that $(\alpha, g)$ and $(\alpha_1, g_1)$ represent the same element of $\int^{[t]\in \Fin} X^t \wedge G_{k,l}([t])$. Without loss of generality we may assume that the functions $\alpha\colon [t]\to X$ and $\alpha_1\colon [t_1]\to X$ are injective. Indeed, suppose for example that $\alpha$ is not injective. Then $\alpha$ can be factored as a surjection followed by injection, say $[t]\twoheadrightarrow [t']\stackrel{\alpha'}{\hookrightarrow} X$. Let $g'$ be the image of $g$ under the map $G_{k,l,m}([t])\to G_{k,l,m}([t'])$. Then $(\alpha', g')$ represents the same element as $(\alpha, g)$ in the coend.

Assuming $\alpha$ and $\alpha_1$ are injective, let $f$ be the common image of $(\alpha, g)$ and $(\alpha_1, g_1)$ in $G_{k,l,m}(X)$. Then there is a unique subset $F\subset X\setminus\{*\}$ such that $f$ factors as in lemma~\ref{lem: factorization}. By the minimality of $F$, $F$ is contained in the image of $\alpha$ and in the image of $\alpha_1$. Choose a pointed bijection $\alpha'\colon [t_0]\stackrel{\cong}{\to}F\cup \{*\}$. This bijection factors through $\alpha$ and $\alpha_1$.
By slight abuse of notation, let use $\alpha'$ to denote also the element of $X^{t_0}$ that is the composed map $[t_0]\stackrel{\alpha'}{\to} F\cup \{*\} \hookrightarrow X$. Putting it all together we obtain a commutative diagram

\begin{diagram}
& \tC_0(X, M_l) & \rTo^{\alpha^*} & M_l^t & & \\
& \dTo^{\alpha_1^*} &\rdOnto^{\alpha'^*} \rdTo(5,2)^f & \dTo  & \\
& M_l^{t_1} & \rTo & M_l^{t_0} & \rInto^{f'} & &  M_k&.
\end{diagram}
It follows that $(\alpha, g)$ and $(\alpha_1, g_1)$ are both mapped to the same element $(\alpha', f')$ of the coend $\int^{[t]\in \Fin} X^t \wedge G_{k,l,m}([t])$.


We have shown that the assembly map~\eqref{eq: assembly} is a bijection. Its domain is compact and its codomain is Hausdorff, so it is a homeomorphism.
%
%
%
%
\end{proof}

\begin{cor}\label{cor: l le k}
If $l>k$ then $\SSS^{k,l}\simeq *$.
\end{cor}
\begin{proof}
It follows immediately from Proposition~\ref{prop:SC inj Mlt Mk} that if $l>k$ then the restriction of $G_{k,l}$ to $\Fin$ is trivial, i.e., $G_{k,l}([t])\cong *$ for every finite pointed set $[t]$. By Theorem~\ref{t:polynom inverse} $G_{k,l}$ is equivalent to the homotopy left Kan extension of the restriction of $G_{k,l}$ to $\Finsk$, so $G_{k,l}\simeq *$. Since $\mathbb{S}^{k,l}$ is the stabilization of $G_{k,l}$, it follows that $\mathbb{S}^{k,l}\simeq *$ as well.
\end{proof}

\subsection{The first stage of the rank filtration}\label{section: first stage}
The functor $G_{k,l,1} $ is especially well behaved.
The following lemma is an easy consequence of Theorem  \ref{t:polynom inverse} taken with $j=1$. We also give a direct proof
\begin{lem}
Let $l,k\geq 1$ and let  $X$  be a finite pointed CW-complex. Then the assembly map
$$a_X\colon X \wedge G_{k,l,1}(S^0) \to G_{k,l,1}(X) $$
is a homeomorphism.
\end{lem}
\begin{proof}
Since $X \wedge G_{k,l,1}(S^0)$ is compact and $G_{k,l,1}(X) $  is Hausdorff it is enough to show that $a_X$ is a bijection. Since $M_l$ is simple we have $G_{k,l}(S^0) = \Csep^{\mathrm{inj}}(M_l,M_k)_{+} $ where
$\Csep^{\mathrm{inj}}(M_l,M_k) \subset \Csep(M_l,M_k)$ is the space of injective $C^*$-algebra maps.
similarly we get  $G_{k,l,1}(S^0)  = \Csep^{,1}(M_l,M_k)_{+} $ where $\Csep^{,1}(M_l,M_k) \subset \Csep^{\mathrm{inj}}(M_l,M_k)$ is the space of maps $f$ with $\mathrm{rank}(f) = 1$.
We get that $(X \wedge G_{k,l,1}(S^0))\smallsetminus \{\ast\} = (X \smallsetminus \{\ast\}) \times \Csep^{,1}(M_l,M_k)$ and the map $a_X(x_0,f)\in G_{k,l,1}(X)  \subset G_{k,l}(X)$ is the composition
$$\tC_0(X, M_l) \xrightarrow{x_0^*} M_l  \xrightarrow{f} M_k.$$
The injectivity now follows from the uniqueness of the factorisation in lemma~\ref{lem: factorization} and the surjectivity from the  existence  in lemma \ref{lem: factorization}  and  lemma \ref{l:lele}.
\end{proof}
\begin{cor}
Let $l,k\geq 1$, the natural map
$$ \Sigma^{\infty}G_{k,l,1}(S^0) \to \SSS^{k,l,1}$$
is an equivalence.
\end{cor}

Let $\mathbb{P}\Inj$ be the topologically enriched  symmetric  monoidal  category of finite positive dimensional  Hilbert spaces  and isometric embeddings up to scalar. That is, up to isomorphism  objects are given by $\mathbb{C}^k$ for $k \in \mathbb{Z}_{\geq 1}$ and $$\mathbb{P}\Inj (\mathbb{C}^l , \mathbb{C}^k) = U(k)/ (U(1) \times U(k-l)). $$
The symmetric monoidal structure is given by  the tensor product.
We have a topologically  enriched symmetric monoidal functor
$$\End\colon\mathbb{P}\Inj  \to  \Csep,$$
that sends the Hilbert space $V$ to the $C^*$-algebra $\End(V)$ of linear maps $V\to V$ and the embedding $i\colon V \to W$ is sent to the $*$-homomorphism $\End(V) \to \End(W)$ sending $A \in \End(V)$ to $i \circ A \circ i^{-1}\circ p \in \End(W)$, where $p\colon W \to \im(i)$ is the orthogonal projection. The monoidal coherence maps of $\End$ are given by the natural isomorphisms $$\End(V_1)\otimes \End(V_2) \xrightarrow{\sim} \End(V_1 \otimes V_2)$$
and
$$\mathbb{C} \xrightarrow{\sim} \End(\mathbb{C}).$$
\begin{lem}\label{l:pinj}
The map
$$\mathbb{P}\Inj(\mathbb{C}^l,\mathbb{C}^k)_{+} \to  \Csep(M_l,M_k)  = G_{k,l}(S^0)$$
induced by the functor $\End$ is an embedding with image $G_{k,l,1}(S^0)$.
\end{lem}
\begin{proof}
We first show that the map is surjective on $G_{k,l,1}(S^0)$.
Let $i\colon \mathbb{C}^l \to \mathbb{C}^k $ be an isometric embedding.  The map $f_i = \End(i) \in \Csep(M_l,M_k)$ clearly satisfies $V_{f} = \mathrm{Im}(i) \subseteq \mathbb{C}^k$ and thus $\mathrm{rank}(f) = \frac{\dim V_{f_i}}{l} = \frac{l}{l} = 1$. On the other hand  if $f_i\in G_{k,l,1}(S^0)$,  then either $f_i=0$ and thus is in the image of the base point or $\dim V_{f_i} = 1$. In the case  $\dim V_{f_i} = 1$ we get that the map $f$ factors as an isomorphism followed by an injection.
$$M_l  \xrightarrow{\sim} \End(V_{f_i}) \to M_k $$
subjectivity now follows from the fact that   every automorphism  of  $M_l$ is inner.
We now prove injectivity. First since $V_{f_i} = \mathrm{Im}(i) $, $f_i$ determines $\mathrm{Im}(i)$.
We are thus reduced to show that if two embeddings $i,j \colon \mathbb{C}^l \to \mathbb{C}^k$ have the same image $V$ and the induced maps $M_l \to \End(V)$ are the same then $i$ and $j$ differ by a scalar. This follows from the fact that  the center of $M_l$ is exactly the scalar matrices.
We thus get a continuous  bijection $\mathbb{P}\Inj(\mathbb{C}^l,\mathbb{C}^k)_{+} \to  G_{k,l,1}(S^0)$. Since
it has compact domain and Hausdorff  target, it is a homeomorphism.
\end{proof}
Applying the topological nerve to $\mathbb{P}\Inj$ we get a symmetric monoidal $\infty$-category $\mathbb{P}\Inj_{\infty}$. We thus get that $\End$ induces a symmetric monoidal functor
$$\widetilde{\End} \colon \mathbb{P}\Inj_{\infty}^{\mathrm{op}} \to \NCW. $$
Composing with the symmetric monoidal functor $\Sigma^{\infty} \colon \NCW \to \NSp$ we get a symmetric monoidal functor
$$\Sigma^{\infty} \circ \widetilde{\End} \colon \mathbb{P}\Inj_{\infty}^{\mathrm{op}} \to \NSp. $$

For a closed symmetric monoidal $\infty$-category $\cM$ denote by $\Cat^{\otimes}_\cM$ the $\infty$-category of symmetric monoidal $\cM$-enriched $\infty$-categories.
If $\cM$ and $\cN$ are closed symmetric monoidal $\infty$-categories and $a \colon \cM \to \cN$ is a symmetric monoidal functor which admits a right adjoint $b$, then we have an induced adjunction
$$a_! \colon\Cat^{\otimes}_\cM \rightleftarrows \Cat^{\otimes}_\cN\noloc b_!.$$

We shall especially use the case where $a = \Sigma^{\infty}_+ \colon \cS \to \Sp$. Using the identification  $\Cat^{\otimes}_\cS \cong \Cat^{\otimes}$, we obtain an adjunction
$$(\Sigma_+^{\infty})_! \colon \Cat^{\otimes} \leftrightarrows  \Cat^{\otimes}_{\Sp} \noloc (\Omega^{\infty})_!$$
Given a symmetric monoidal $\infty$-category $\mathcal{C} \in \Cat^{\otimes}$ we denote
$$\mathcal{C}^{\Sp} :=(\Sigma^{\infty}_+ )_!(\mathcal{C}) \in \Cat^{\otimes}_\Sp.$$

As a stable $\infty$-category $\NSp$ is naturally $\Sp$-enriched. Thus we have natural isomorphisms
$$\Map_{\Cat^{\otimes}_{\Sp}} ((\mathbb{P}\Inj^{\Sp}_{\infty})^{\mathrm{op}},\NSp)\simeq \Map_{\Cat^{\otimes}_{\Sp}}((\Sigma_+^{\infty})_! (\mathbb{P}\Inj_{\infty}^{\mathrm{op}}),\NSp)\simeq $$ $$\Map_{\Cat^{\otimes}}(\mathbb{P}\Inj_{\infty}^{\mathrm{op}}, (\Omega^{\infty})_!\NSp)\simeq \Map_{\Cat^{\otimes}}(\mathbb{P}\Inj_{\infty}^{\mathrm{op}}, \NSp).$$
We denote the mate of $\Sigma^{\infty} \circ \widetilde{\End}\in \Map_{\Cat^{\otimes}} (\mathbb{P}\Inj_{\infty}^{\mathrm{op}}, \NSp)$ under this adjunction by
$$\widetilde{\mathrm{E}} \colon (\mathbb{P}\Inj_{\infty}^{\Sp})^\mathrm{op} \to \NSp.$$
This is a symmetric monoidal $\Sp$-enriched functor.

\begin{lem}\label{l:big}
We  get the following commutative diagram for every $k,l \geq 1$

$$\xymatrix{
\Map_{(\mathbb{P}\Inj_{\infty}^{\Sp})^\mathrm{op}}(\CC^k , \CC^l) \ar[d]^{\sim} \ar[rr]^-{\tilde{\mathrm{E}}} && \Map_{\NSp}(\Sigma^{\infty}M^k ,\Sigma^{\infty}M^k)  \ar[dd]^{\sim} \\
\Sigma^{\infty}_+ \Map_{\mathbb{P}\Inj_{\infty}}(\CC^l , \CC^k) \ar[d]^{\sim}\ar[r]^{\End}& \Sigma^{\infty}\Map_{\NCW}(M^k,M^l)\ar[d]^{\sim}\ar[ru]&\\
\Sigma^{\infty}G_{k,l,1}(S^0)\ar[rd]_{\sim}  \ar[r]&  \Sigma^{\infty}G_{k,l}(S^0) \ar[r]  & \partial_1(G_{k,l})\\
 & \partial_1(G_{k,l,1})  \ar[ru] \\
}$$
\end{lem}
\begin{proof}
The commutation of the lower triangle and the right trapezoid is clear. The left square is a consequence  of lemma \ref{l:pinj}. To see the commutation of the top trapezoid consider the digram in $\Cat_{\Sp}^{\otimes}$.
$$
\xymatrix{
 (\mathbb{P}\Inj_{\infty}^{\Sp})^\mathrm{op} \ar[rd]\ar[r]& (\Sigma^{\infty})_!( \NCW) \ar[d] \\
 & \NSp
 }.$$
\end{proof}

%

\section{Connection with topological $K$-theory}\label{section: k-theory}
In this section we show that for every fixed $l$, the spectrum $\SSS^{\infty, l}:=\colim_{k\to \infty} \SSS^{k,l}$, is equivalent to the connective $K$-theory spectrum $ku$. We will use this observation to show how the spectra $\SSS^{\infty, l}$ together represent $K$-theory of noncommutative CW-complexes.

Recall that for a pointed finite set $[t]$,
\[
G_{k,l}([t])=\Csep(M_l^t, M_k)\cong \bigvee_{(m_1, \ldots, m_t)} {\Inj(\CC^{(m_1+\cdots+m_t)l}, \CC^k)/_{\prod_{i=1}^t U(m_i)}}_+.
\]
The inclusions of algebras $M_k\to M_{k+1}$ that send a matrix $a$ to $a\oplus 0$ induce a sequence of natural transformations, for each fixed $l$, $$\cdots \to G_{k,l}\to G_{k+1,l}\to \cdots.$$
\begin{define}
For each fixed $l$, define the functor $G_{\infty, l}:=\colim_{k\to \infty}G_{k, l}$. Here by $\colim$ we mean strict rather than homotopy colimit.
\end{define}
\begin{lem}\label{lemma: Ginfty}
$G_{\infty, l}$ is equivalent to the homotopy colimit $\hocolim_{k\to \infty} G_{k,l}$. Moreover there is a homeomorphism (where as usual the wedge sum is indexed on the set of non-zero $t$-tuples of non-negative integers)
\[
G_{\infty,l}([t])= \bigvee_{(m_1, \ldots, m_t)} {\Inj(\CC^{(m_1+\cdots+m_t)l}, \CC^\infty)/_{\prod_{i=1}^t U(m_i)}}_+
\]
and a homotopy equivalence
\[
G_{\infty,l}([t])\simeq \bigvee_{(m_1, \ldots, m_t)} BU(m_1)\times \cdots\times BU(m_t)_+.
\]
\end{lem}
\begin{proof}
The map
\[
\Inj(\CC^{(m_1+\cdots+m_t)l}, \CC^k)/_{\prod_{i=1}^t U(m_i)} \to \Inj(\CC^{(m_1+\cdots+m_t)l}, \CC^{k+1})/_{\prod_{i=1}^t U(m_i)}
\]
is an inclusion of a submanifold, and in particular a cofibration. It follows that the colimit is equivalent to the homotopy colimit. As $k$ goes to $\infty$, the colimit of $G_{k,l}$ is, by definition, homeomorphic to
\[ \bigvee_{(m_1, \ldots, m_t)} {\Inj(\CC^{(m_1+\cdots+m_t)l}, \CC^\infty)/_{\prod_{i=1}^t U(m_i)}}_+.\]
Note that $\Inj(\CC^{(m_1+\cdots+m_t)l}, \CC^\infty)$ is a contractible space with a free action of $\prod_{i=1}^t U(m_i)$, and moreover the quotient map $$\Inj(\CC^{(m_1+\cdots+m_t)l}, \CC^\infty)\to \Inj(\CC^{(m_1+\cdots+m_t)l}, \CC^\infty)/_{\prod_{i=1}^t U(m_i)}$$ is a fiber bundle. It follows that the quotient space is equivalent to the classifying space $BU(m_1)\times \cdots\times BU(m_t)$.
\end{proof}
Recall $\SSS^{k,l}$ is the stabilization of $G_{k,l}$. We define $\SSS^{\infty, l}$ accordingly.
\begin{define}
The spectrum $\SSS^{\infty, l}$ is defined as follows
\[
\SSS^{\infty, l}:=\colim_{k\to \infty} \SSS^{k,l}.
\]
\end{define}
Since stabilization commutes with homotopy colimits of pointed functors, $\SSS^{\infty, l}$ is the stabilization of $G_{\infty, l}$.
It is worth noting that the homotopy type of $\SSS^{\infty,l}$ is independent of $l$. Indeed, the following statement is an immediate consequence of Lemma~\ref{lemma: Ginfty}.
\begin{cor}\label{cor: same}
A choice of an inclusion $\CC\hookrightarrow \CC^l$ induces an equivalence of $\Gamma$-spaces $G_{\infty, l}\to G_{\infty ,1}$.
\end{cor}
In fact, the spectrum $\SSS^{\infty, l}$ is homotopy equivalent to the connective $K$-theory spectrum $ku$ for each $l$. We record this observation in a lemma.
\begin{lem}\label{lemma: allku}
For each $l$ the spectrum $\SSS^{\infty, l}$ is homotopy equivalent to $ku$.
\end{lem}
\begin{proof}
The $\Gamma$-space $G_{\infty, 1}$ is equivalent to the $\Gamma$-space constructed by Segal in~\cite[Section 2]{Se1}. It is, in Segal's terminology, a {\it special} $\Gamma$-space. This means that for any two pointed finite sets $[s], [t]$, the natural map $G_{\infty, 1}([s]\vee [t])\to G_{\infty, 1}([s])\times G_{\infty, 1}([t])$ is an equivalence. It follows from Lemma~\ref{lemma: Ginfty} that
\[
G_{\infty, 1}([1])\cong \bigvee_{m=1}^\infty BU(m)_+.
\]
So $\SSS^{\infty, 1}$ is the spectrum associated with the group-completion of $\bigvee_{m=1}^\infty BU(m)_+$, which is well-known to be $ku$. By corollary~\ref{cor: same}, it follows that $\SSS^{\infty, l}\simeq ku$ for all $l$.
\end{proof}
It follows that for each $k, l$ there is a natural map $\SSS^{k,l}\to \SSS^{\infty, l}\simeq ku$. We will show later, after we analyze the subquotients of the rank filtration, that this map induces an isomorphism on $\pi_0$ (Lemma~\ref{lem: over ku}). In particular, it is not trivial.

\subsubsection*{$K$-theory of noncommutative CW-complexes.}\label{subsection: K-theory}
We will now discuss how the spectra $\SSS^{\infty, l}$ can be used to represent the $K$ theory functor on the category~$\NSp$. In this subsection we consider $\cM$ as an $\infty$-category and use the $\infty$-categorical picture.

Consider the enriched Yoneda embedding of $\cM^{\op}$ (which is an $\Sp^{rev}$-functor):
\[
Y\colon\cM^{\op}\to \Fun_\Sp(\cM,\Sp).
\]
The sequence of algebras $\cdots\to M_k\to M_{k+1}\to \cdots$ gives rise to a direct sequence in $\cM^{\op}$: \[\cdots\to  \Sigma^\infty_{\NC} M_k\to  \Sigma^\infty_{\NC} M_{k+1}\to \cdots\]
Applying the Yoneda embedding $Y$ to this sequence, we obtain a sequence of spectral functors $Y(\Sigma^\infty_{\NC} M_k)\in \Fun_\Sp(\cM,\Sp)$, that are characterized by the following property
\[
Y(\Sigma^\infty_{\NC} M_k)(\Sigma^\infty_{\NC} M_l)\simeq \Hom_{\cM}(\Sigma^\infty_{\NC} M_k, \Sigma^\infty_{\NC} M_l) = \SSS^{k,l}.
\]
The homotopy colimit of this sequence is a spectral functor that will represent connective $K$-theory. Let us give this functor a name:
\begin{define}
We define the spectrally enriched functor $\ku\in\Fun_\Sp(\cM,\Sp)$ to be the following colimit:
$$\ku=\colim_{k\to\infty} Y(\Sigma^\infty_{\NC} M_k).$$
\end{define}
Note that for every $l$ we have the following equivalences, the last of which follows from Lemma~\ref{lemma: allku}. The last equivalence justifies the notation $\ku$ for this functor.
$$\ku(\Sigma^\infty_{\NC} M_l):=\colim_{k\to\infty}\Hom_{\NSp}(\Sigma^\infty_{\NC} M_k, \Sigma^\infty_{\NC} M_l)=\SSS^{\infty, l}\simeq ku.$$

We know by the main result of~\cite{ABS1} that we can identify $\NSp$ with $P_{\Sp}(\cM)$, the category of contravariant spectrally enriched functors from $\cM$ to $\Sp$. This implies that {\it covariant} enriched functors from $\cM$ to $\Sp$, such as $\ku$, can be used to define homology theories on $\NSp$, and therefore also on $\NCW$.

Indeed suppose $h\colon\cM\to \Sp$ is a covariant $\Sp$-functor. Then $h$ determines an $\Sp$-tensored functor
$$h\wedge_{\cM} (-):\NSp\to\Sp$$
by the universal property of the enriched Yoneda embedding, such that there is a natural equivalence
\[
h(k)\wedge_{k\in \cM} \SSS^{k,l} \simeq h(\Sigma^\infty_{\NC} M_l).
\]
The universal property also tells us that this induces an equivalence between the $\infty$-category of $\Sp$-tensored functors $\NSp\to\Sp$ and the $\infty$-category of $\Sp$-functors $\cM\to \Sp$. By taking homotopy groups $\pi_*(h\wedge_{\cM} (-))$ one obtains a generalized homology theory on noncommutative CW-spectra (see \cite[Definition 4.1]{BJM}). 

In particular, let us take $h=\ku$. We interpret the property that $\ku(\Sigma^\infty_{\NC} M_l)\simeq ku$ for all $l$ as saying that $\ku$ represents a version of connective $K$-theory. We have a further enhancement of this fact.
\begin{lem}\label{lem: bu}
The functor $\ku$ takes values in the category of $ku$-module spectra.
\end{lem}
\begin{proof}
The monoidal structure on $\cM$ gives rise to maps $\SSS^{k',l'}\wedge \SSS^{k,l} \to \SSS^{kk', ll'}$.  Fixing $l'=1$ we obtain maps $\SSS^{k',1}\wedge \SSS^{k,l} \to \SSS^{kk', l}$, natural in $l$. Taking limits as $k, k'\to \infty$ we obtain maps, still natural in $\Sigma^\infty_{\NC} M_l$
\[
\ku(\Sigma^\infty_{\NC} M_1)\wedge \ku(\Sigma^\infty_{\NC} M_l)\to \ku(\Sigma^\infty_{\NC} M_l).
\]
Identifying $\ku(\Sigma^\infty_{\NC} M_1)$ with $ku$, this map endows $\ku(\Sigma^\infty_{\NC} M_l)$ with the structure of a $ku$-module, functorial in $\Sigma^\infty_{\NC} M_l$.
\end{proof}

The $\infty$-category of $ku$-module spectra $ku\mathrm{-mod}$ is stable and thus $\Sp$-tensored. In light of Lemma \ref{lem: bu}, we may now view $\ku$ as an $\Sp$-functor $\ku\colon \cM\to ku\mathrm{-mod}$. Thus, by the universal property of the enriched Yoneda embedding, $\ku$ determines an $\Sp$-tensored functor
$$\ku\wedge_{\cM} (-):\NSp\to ku\mathrm{-mod}.$$
This allows for the following definition:
\begin{define}
For $X, Y\in \NSp$ we define
\[kk(X, Y):=[\ku \wedge_{\cM} X, \ku \wedge_{\cM} Y]_{ku\mathrm{-mod}},
\]
\[kk_*(X, Y)=\pi_*\Map_{ku\mathrm{-mod}}(\ku \wedge_{\cM} X, \ku \wedge_{\cM} Y)=kk(\Sigma^*X,Y).
\]
\end{define}

It follows from Lemma \ref{lem: bu} that for all $k,l\geq 1$
\[
kk_*(\Sigma^\infty_{\NC} M_k, \Sigma^\infty_{\NC} M_{l})\simeq \pi_*\Map_{ku\mathrm{-mod}}(ku, ku)\simeq \pi_*ku.
\]
Recall from Remark \ref{remark: commutative} that the category of commutative spectra is identified with the category of spectral presheaves on the full $\Sp$-enriched subcategory of $\cM$ containing the object $\Sigma^\infty_{\NC} M_1$. A commutative spectrum is made noncommutative by means of an $\Sp$-enriched left Kan extension. Suppose that $X$ is a commutative spectrum, which we also may consider as a noncommutative spectrum by Kan extension. Standard adjunction implies that there is a natural equivalence
\[
\ku \wedge_{\cM} X \simeq ku \wedge X,
\]
where the symbol $\wedge$ on the right hand side denotes the usual smash product of spectra. It follows that if $X$ and $Y$ are commutative spectra then there is an equivalence
\[
kk_*(X, Y)\simeq \pi_* \Map_{\Sp}(X, ku \wedge Y).
\]
In the case when $X$ and $Y$ are suspension spectra of pointed finite CW-complexes, this is essentially~\cite[Proposition 3.1]{DM}. Thus $kk$ defined above is a natural extension for the connective bivariant $K$-theory of Dadarlat and McClure. This also suggests that some results from [loc. cit.] may be generalized from finite CW-complexes to finite noncommutative CW-complexes. This will be addressed in future papers.

%
%

We conclude by remarking that since the functor $\ku$ takes values in modules over $ku$, one can invert the Bott element and get a functor representing non-connective $K$-theory.

\subsubsection{Connection to the rank filtration of $ku$}
It follows immediately from equation~\eqref{eq: rank filtration}, that the filtration of $G_{k,l}$ by $G_{k,l,m}$ interacts well with the stabilization map $G_{k,l}\to G_{k+1, l}$ that we considered in the previous section. To be more specific, there is a commutative diagram
\[
\begin{array}{ccccccccc}
\cdots & \to & G_{k, l, m} & \to & G_{k, l, m+1} & \to & \cdots & \to & G_{k,l} \\
 &  & \downarrow & & \downarrow & & \cdots & & \downarrow\\
\cdots & \to & G_{k+1, l, m} & \to & G_{k+1, l, m+1} & \to & \cdots & \to & G_{k+1,l} \\
 &  & \downarrow & & \downarrow & & \cdots & & \downarrow\\
 &  & \vdots & & \vdots & & \cdots & & \vdots \\
\cdots & \to & G_{\infty, l, m} & \to & G_{\infty, l, m+1} & \to & \cdots & \to & G_{\infty,l}.
\end{array}
\]
On finite sets, $G_{\infty, l, m}$ is given, at least up to homotopy, by the following formula
\[
G_{\infty, l, m}([t])\simeq  \bigvee_{\underset{ m_1+\cdots + m_t\le m\}}{\{(m_1, \ldots, m_t)\mid}} BU(m_1)\times \cdots \times BU(m_t)_+
\]
We conclude that the rank filtration of $G_{k,l}$ induces a compatible rank filtration of $G_{\infty, l}$. Upon passing to stabilization, the rank filtration of $G_{\infty, l}$ induces a filtration of $ku$ by a sequence of spectra:
\[
\cdots  \to  \SSS^{\infty, l, m}  \to  \SSS^{\infty, l, m+1}  \to  \cdots  \to  \SSS^{\infty,l}\simeq ku.
\]
The latter filtration is the familiar rank filtration of the $K$-theory spectrum $ku$ studied, for example, in~\cite{AL-Fundamenta} (we remark that $ku$ is denoted $bu$ in [loc. cit]). Thus the rank filtration of $\SSS^{k,l}$ is a lift of the classical rank filtration of $ku$.

\section{Subquotients of the rank filtration}\label{ss: subquotients}
In this section we investigate the subquotients of the rank filtration. We show that on the level of the functors $G_{k,l}$ the subquotients of the rank filtration have a presentation as a homotopy coend over the category $\Epi$ of finite sets and surjections, as opposed to the category $\Fin$ of pointed sets and all pointed functions.

\begin{define}
Let $G_{k,l}^{m}\colon \CW\to \Top$ be the quotient functor $$G_{k,l}^{m}:=G_{k,l,m}/G_{k,l,m-1}.$$
Similarly, let $\SSS^{k,l}_m$ be the homotopy cofiber of the map $\SSS^{k,l,m-1} \to \SSS^{k,l,m}$.
\end{define}
It is not hard to check that for any $X\in \CW$ the map $G_{k,l,m-1}(X)\to G_{k,l,m}(X)$ is a cofibration in $\Top$. Thus, $G_{k,l}^{m}$ the levelwise homotopy cofiber of the map $G_{k,l,m-1}\to G_{k,l,m}$ in the category of pointed continuous functors $\Fun_*(\CW,\Top)$. It follows from \cite[Lemma 2.16]{ADL} that $G_{k,l}^{m}$ is also the homotopy cofiber of the map $G^{k,l,m-1} \to G^{k,l,m}$ in the projective model structure on $\Fun_*(\CW,\Top)$, and thus also in the stable model structure $\SpM$.

Since strict (homotopy) left Kan extension commutes with strict (homotopy) cofiber, we get from Theorem \ref{t:polynom inverse} that the functor $G_{k,l}^{m}$ is equivalent to both the strict and the derived left Kan extension of its restriction to ${\Fin^{\le j}}$ for any $\infty\ge j\ge \min(m, \lfloor\frac{k}{l}\rfloor)$. Thus, for any finite pointed CW-complex $X$, we have the folowing formula:
\begin{equation}\label{eq: G_kl^m}
G_{k,l}^{m}(X)\cong\int^{[t]\in \Fin^{\le j}}_{\mathrm s}X^t\wedge G_{k,l}^{m}([t])\simeq\int^{[t]\in \Fin^{\le j}}_{\mathrm h}X^t\wedge G_{k,l}^{m}([t]).
\end{equation}
Since stabilization commutes with homotopy cofibers, $\SSS^{k,l}_m$ is equivalent to the stabilization of $G_{k,l}^{m}$.


It follows immediately from equation~\eqref{eq: rank filtration} that on objects $G_{k,l}^{m}$ is given as follows
\begin{equation}\label{eq: subquotient}
G_{k,l}^{m}([t])= \bigvee_{\underset{ m_1+\cdots + m_t= m\}}{\{(m_1, \ldots, m_t)\mid}} {\Inj(\CC^{(m_1+\cdots+m_t)l}, \CC^k)/_{\prod_{i=1}^t U(m_i)}}_+
\end{equation}
Which also can be written as
\[
G_{k,l}^{m}([t]) = \bigvee_{\underset{ m_1+\cdots + m_t= m\}}{\{(m_1, \ldots, m_t)\mid}} {\Inj(\CC^{ml}, \CC^k)/_{\prod_{i=1}^t U(m_i)}}_+
\]
To understand the effect of $G_{k,l}^{m}$ on morphisms, let $\alpha\colon [t]\to [s]$ be a pointed function. If $|\alpha_*(m_1, \ldots, m_t)|< m$, then $G_{k,l}^{m}(\alpha)$ takes the summand of $G_{k,l}^{m}([t])$ corresponding to $(m_1, \ldots, m_t)$ to the basepoint. If $|\alpha_*(m_1, \ldots, m_t)| = m$, then $G_{k,l}^{m}(\alpha)$ takes the corresponding summand of $G_{k,l}^{m}([t])$ to the summand of $G_{k,l}^{m}([s])$ indexed by $\alpha_*(m_1, \ldots, m_t)$ by the map~\eqref{eq: functor}.

One attractive property of $G_{k,l}^{m}$ is that it has a more compact coend formula than the one given in equation \ref{eq: G_kl^m} (see Proposition \ref{prop: little coend}). In this formula the category $\Fin$ is replaced with the smaller category $\Epi$ of non-empty unpointed sets and surjections. For $k\le \infty$ let $\Epi^{\le k}$ be the category of non-empty finite sets of cardinality at most $k$ and epimorphisms between them. In the case $k=\infty$, this is the category of all non-empty finite sets and surjections, and we denote it simply $\Epi$.

If $\alpha\colon \underline{t}\twoheadrightarrow \underline{s}$ is a morphism in $\Epi$, and $(m_1, \ldots, m_t)$ is a multiset, then we may define $\alpha_*(m_1, \ldots, m_t)= (n_1, \ldots, n_s)$ in the usual way, by saying that $n_j=\sum_{i\in\alpha^{-1}(j)}m_i$. Note that in this case there is an equality $m_1+\cdots+m_t=n_1+\cdots+n_s$. Note also that if $m_i>0$ for all $i$ then $n_j>0$ for all $j$.
\begin{define}\label{definition: Gred}
Let $\uTop$ be the category of unpointed topological spaces. Let $\Gred_{k,l}^m\colon \Epi \to \uTop$ be the following functor. On objects, it is defined by the following formula
\[
\Gred_{k,l}^{m}(\underline{t}) = \coprod_{\underset{ m_i>0, m_1+\cdots + m_t= m\}}{\{(m_1, \ldots, m_t)\mid}} {\Inj(\CC^{ml}, \CC^k)/_{\prod_{i=1}^t U(m_i)}}
\]
On morphisms, $\Gred_{k,l}^m$ is defined similarly to $G_{k,l}$ and $G_{k,l}^m$. Given a surjection $\alpha\colon \underline{t}\twoheadrightarrow \underline{s}$, the summand indexed by $(m_1, \ldots, m_t)$ is mapped to the summand indexed by $\alpha_*(m_1, \ldots, m_t)$ by the same map as in~\eqref{eq: functor}.
\end{define}
We will make much use of the functor ${\Gred_{k,l}^m}_+\colon \Epi \to \Top$, which is obtained by adding a disjoint basepoint to $\Gred_{k,l}^m$. We introduced an unpointed version of the functor because it will be convenient to have it at some point.
\begin{rem}\label{remark: decomposition}
For a multi-set $(m_1, \ldots, m_t)$ define its support to be the set $A=\{i\mid m_i>0\}\subseteq \underline{t}$. Notice that for any subset $A\subseteq \underline{t}$ there is a natural way to identify ${\Gred_{k,l}^{m}(A)}_+$ with a wedge summand of $G_{k,l}^m([t])$. Namely, ${\Gred_{k,l}^{m}(A)}_+$ is identified, on the right hand side of~\eqref{eq: subquotient} as the wedge sum corresponding to indices $(m_1,\ldots, m_t)$ whose support is exactly $A$. With this identification, there is a homeomorphism
\[
G_{k,l}^m([t])\cong \bigvee_{A\subseteq\underline{t}}  {\Gred_{k,l}^{m}(A)}_+
\]
Moreover, the functoriality in $[t]$ is defined on the right hand side as follows. Suppose $\alpha\colon[t]\to [s]$ is a pointed function. Suppose $A\subseteq \underline{t}$. If $\alpha$ sends some element of $A$ to the basepoint of $[s]$, then the corresponding summand $\Gred_{k,l}^{m}(A)$ is sent to the basepoint. Otherwise, this summand is sent to $\Gred_{k,l}^{m}(\alpha(A))$ using the surjection $A\twoheadrightarrow \alpha(A)$ defined by $\alpha$.
\end{rem}

Consider the functor
\[\begin{array}{ccc}
 \Fin \times \Epi^{\op} &\to& \Top \\
 ([t], u) & \to & [t]^{\wedge u}
 \end{array}
\]
One can think of this functor as a $\Fin-\Epi$-bimodule. It is often used to establish connections between categories of $\Fin$-modules and $\Epi$-modules. We have the following proposition:
\begin{prop}\label{prop: reduction}
There is a homeomorphism and an equivalence, natural in $[t]$ ranging over $\Fin$:
\[
G_{k,l}^m([t])\cong \int^{u\in \Epi}_{\mathrm s} [t]^{\wedge u} \wedge {\Gred_{k,l}^m(u)}_+ \simeq \int^{u\in \Epi}_{\mathrm h} [t]^{\wedge u} \wedge {\Gred_{k,l}^m(u)}_+
\]
\end{prop}
\begin{proof}
First of all, let us construct a natural map. Fix a surjective function $u_1\twoheadrightarrow u_2$. One has a map
\[
[t]^{\wedge u_2} \wedge {\Gred_{k,l}^m(u_1)}_+\to G_{k,l}^m([t]).
\]
The map is defined as follows. First, the surjection $u_1\twoheadrightarrow u_2$ induces a map $ \Gred_{k,l}^m(u_1)\twoheadrightarrow  \Gred_{k,l}^m(u_2)$, and therefore $[t]^{\wedge u_2} \wedge {\Gred_{k,l}^m(u_1)}_+\to  [t]^{\wedge u_2} \wedge {\Gred_{k,l}^m(u_2)}_+$. Second, there is a bijection of sets $[t]^{\wedge u_2}\cong {\underline{t}^{u_2}}_+$, so a non basepoint of this set is a map $f\colon u_2\to \underline{t}$. This defines a map $\Gred_{k,l}^m(u_2)\to \Gred_{k,l}^m(f(u_2))$. Finally, there is an inclusion of ${\Gred_{k,l}^m(f(u_2))}_+$ as a wedge summand of $G_{k,l}^m([t])$, as in Remark \ref{remark: decomposition}.

The map is natural in the variable $[t]$ (exercise for the reader). 

The following diagram commutes because $\Gred_{k,l}^m$ is functorial with respect to surjections.
\[
\begin{array}{ccc}
[t]^{\wedge u_2} \wedge {\Gred_{k,l}^m(u_1)}_+ &\to & [t]^{\wedge u_2} \wedge {\Gred_{k,l}^m(u_2)}_+\\
\downarrow & & \downarrow \\

[t]^{\wedge u_1} \wedge {\Gred_{k,l}^m(u_1)}_+ & \to & G_{k,l}^m([t])
\end{array}
\]
It follows that there is a natural transformation of functors of $[t]$ (recall that $\int_{\mathrm s}$ denotes strict coend)
\begin{equation}\label{eq: strict assembly}
 \int^{u\in \Epi}_{\mathrm s} [t]^{\wedge u} \wedge {\Gred_{k,l}^m(u)}_+ \to G_{k,l}^m([t]).
\end{equation}
We claim that this map is in fact an isomorphism. To see this notice that for each fixed $t$ there is an isomorphism of functors of $u$
\begin{equation}\label{eq: smash power}
[t]^{\wedge u}={\underline{t}^u}_+ \cong \bigvee_{A\subseteq \underline{t}} \Epi(u, A)_+
\end{equation}
For each $A$, the functor $u\mapsto \Epi(u, A)$ is a representable functor $\Epi^{\op}\to \Top$. By coYoneda lemma, there is an isomorphism
\[
 \int^{u\in \Epi}_{\mathrm s} [t]^{\wedge u} \wedge {\Gred_{k,l}^m(u)}_+ \cong \bigvee_{A\subseteq \underline{t}} {\Gred_{k,l}^m(A)}_+
\]
and the right hand is identified with $G_{k,l}^m([t])$, again as in Remark \ref{remark: decomposition}.

It follows that the map~\eqref{eq: strict assembly} is in fact an isomorphism. On the other hand, Equation~\eqref{eq: smash power} shows that the functor $u\mapsto [t]^{\wedge u}$ is cofibrant in the projective model structure on the functor category $[\Epi^{\op}, \Top]$. It follows that the natural map from the homotopy coend to the strict coend is an equivalence:
\[
\int^{u\in \Epi}_{\mathrm h} [t]^{\wedge u} \wedge {\Gred_{k,l}^m(u)}_+ \stackrel{\simeq}{\to} \int^{u\in \Epi}_{\mathrm s} [t]^{\wedge u} \wedge {\Gred_{k,l}^m(u)}_+.
\]
\end{proof}
As a consequence, we have a simplified coend formula for $G_{k,l}^m(X)$ where $X$ is a CW complex.
\begin{prop}\label{prop: little coend}
Let $X$ be a finite pointed CW-complex. There is a homeomorphism and an equivalence, natural in $X$
\[
G_{k,l}^m(X)\cong \int^{u\in \Epi}_{\mathrm s} X^{\wedge u} \wedge {\Gred_{k,l}^m(u)}_+\simeq \int^{u\in \Epi}_{\mathrm h} X^{\wedge u} \wedge {\Gred_{k,l}^m(u)}_+.
\]
The statement remains true if $\Epi$ is replaced with $\Epi^{\le k}$.
\end{prop}
\begin{proof}
We prove the part for the homotopy coend, and the proof of the strict part is identical.
By the standard coend formula (see equation \ref{eq: G_kl^m}), there is an equivalence
\[
G_{k,l}^m(X)\simeq \int^{[t]\in \Fin}_{\mathrm h} X^{[t]} \wedge G_{k,l}^m([t]).
\]
By Proposition~\ref{prop: reduction}, there is an equivalence
\[
G_{k,l}^m([t]) \simeq \int^{u\in \Epi}_{\mathrm h} [t]^{\wedge u} \wedge {\Gred_{k,l}^m(u)}_+.
\]
It follows that there is an equivalence
\[
G_{k,l}^m(X)\simeq \int^{[t]\in \Fin}_{\mathrm h} X^{[t]} \wedge \left(\int^{u\in \Epi}_{\mathrm h} [t]^{\wedge u} \wedge {\Gred_{k,l}^m(u)}_+\right).
\]
By associativity of coend (``Fubini theorem''), the right hand side is equivalent to
\[
\int^{u\in \Epi}_{\mathrm h} \left(\int^{[t]\in \Fin}_{\mathrm h} X^{[t]} \wedge  [t]^{\wedge u}\right) \wedge {\Gred_{k,l}^m(u)}_+.
\]
It remains to show that there is a natural equivalence
\[
\int^{[t]\in \Fin}_{\mathrm h} X^{[t]} \wedge  [t]^{\wedge u} \simeq X^{\wedge u}.
\]
This is elementary. The argument goes as follows. The set $[t]^{\wedge u}$ is equivalent to the total homotopy cofiber of the cubical diagram $A\mapsto \Fin(A_+, [t])$, where $A$ ranges over subsets of $u$, and the maps are induced by collapsing the complement of a subset to the basepoint. By the coYoneda lemma, $\int^{[t]\in \Fin}_{\mathrm h} X^{[t]} \wedge  \Fin(A_+, [t]) \simeq X^A$. It follows that $\int^{[t]\in \Fin}_{\mathrm h} X^{[t]} \wedge  [t]^{\wedge u}$ is equivalent to the total homotopy cofiber of the cubical diagram  $A\mapsto X^A$, where $A$ ranges over subsets of $u$. The total cofiber is equivalent to $X^{\wedge u}$.
\end{proof}

Our next step is to use Proposition~\ref{prop: little coend} to describe the subquotient spectra $\SSS^{k,l}_m$.
Let $\mathtt{I}\colon \Epi^{\op}\to \Top$ be the (unique) functor defined by
\[
\mathtt{I}(\underline{t})=\left\{\begin{array}{cc} S^0 & t = 1 \\ * & t \ne 1 \end{array}\right.
\]
\begin{lem}\label{lemma: stable coend}
There are equivalences, where $\Epi$ can be replaced with $\Epi^{\le k}$
\[
\SSS^{k,l}_m\simeq \Sigma^\infty\int^{\underline{t}\in\Epi}_{\mathrm h}\mathtt{I}\wedge {\Gred_{k,l}^m}_+ \simeq \int^{\underline{t}\in\Epi}_{\mathrm h}\Sigma^\infty \mathtt{I}\wedge {\Gred_{k,l}^m}_+.
\]
\end{lem}
\begin{proof}
By definition,
$\SSS^{k,l}_m$ is the stabilization of the functor
\[
X\mapsto \int^{u\in \Epi}_{\mathrm h} X^{\wedge u} \wedge {\Gred_{k,l}^m(u)}_+
\]
The right hand side is a weighted homotopy colimit of reduced functors from $\FCW\to \Top$. Since stabilization commutes with such homotopy colimits, it follows that there is an equivalence
\[
\SSS^{k,l}_m\simeq  \int^{u\in \Epi}_{\mathrm h} \partial_1(X^{\wedge u}) \wedge {\Gred_{k,l}^m(u)}_+
\]
where $\partial_1(X^{\wedge u})$ denotes the stabilization of the functor $X\mapsto X^{\wedge u}$.
Observe that $\partial_1(X^{\wedge u})$ is equivalent to $\Sigma^\infty S^0$ if $|u|=1$, and is equivalent to $*$ if $|u|>1$. It follows that the functor $u\mapsto \partial_1(X^{\wedge u})$ is equivalent to $\Sigma^\infty \mathtt{I}$ as a functor $\Epi^{\op}\to \Sp$. The lemma follows.
\end{proof}
Recall once again that ${\Gred_{k,l}^m}_+$ is defined by the following formula,
\begin{multline*}
{\Gred_{k,l}^m(\underline{t})}_+=\bigvee_{\underset{ m_i>0, \Sigma m_i = m\}}{\{(m_1, \ldots, m_t)\mid}}  {\Inj(\CC^{ml}, \CC^k)/_{\prod_{i=1}^t U(m_i)}}_+ \cong  \\ \cong \bigvee_{\underset{ m_i>0, \Sigma m_i = m\}}{\{(m_1, \ldots, m_t)\mid}} U(k)/\left(\prod_{i=1}^t U(m_i) \times U(k-lm)\right)_+.
\end{multline*}
In the special case $l=1, k=m$ we get that
\begin{equation}\label{eq: basic case}
{\Gred_{m,1}^{m}(\underline{t})}_+= \bigvee_{\underset{ m_i>0, \Sigma m_i = m\}}{\{(m_1, \ldots, m_t)\mid}} U(m)/\prod_{i=1}^t U(m_i)_+.
\end{equation}
And in general, there are equivalences
\begin{multline} \label{equation: general reduction}
{\Gred_{k,l}^m(\underline{t})}_+\simeq \bigvee_{\underset{ m_i>0, \Sigma m_i = m\}}{\{(m_1, \ldots, m_t)\mid}} U(k)/U(k-lm)_+\wedge_{U(m)} U(m)/\prod_{i=1}^t U(m_i)_+\simeq \\ \simeq U(k)/U(k-lm)_+\wedge_{U(m)} {\Gred_{m,1}^{m}(\underline{t})}_+.
\end{multline}
It is easily checked that the last equivalence is functorial in $\underline{t}$ and therefore we have an equivalence of functors $\Epi\to \Top$
\[
{\Gred_{k,l}^m}_+\simeq U(k)/U(k-lm)_+\wedge_{U(m)} {\Gred_{m,1}^{m}}_+.
\]
And upon applying lemma~\ref{lemma: stable coend} we obtain an equivalence
\begin{equation}\label{equation: stable reduction}
\SSS^{k,l}_{m}\simeq U(k)/U(k-lm)_+\wedge_{U(m)}\SSS^{m,1}_{m}= \Inj(\CC^{lm}, \CC^k)_+\wedge_{U(m)} \SSS^{m,1}_{m}.
\end{equation}
We remind the reader that $U(m)$ is considered a subgroup of $U(k)$ via the diagonal map $U(m)\hookrightarrow U(lm)$ followed by the inclusions $U(lm)\hookrightarrow U(lm)\times U(k-lm)\hookrightarrow U(k)$. Alternatively, $U(m)$ acts on $\Inj(\CC^{lm}, \CC^k)$ through its obvious action on $\CC^{lm}=\CC^l\otimes \CC^m$.
%
\section{Connection with the complex of direct-sum decompositions}\label{ss: Ln}
Equation~\eqref{equation: stable reduction} reduces the problem of describing $\SSS^{k,l}_{m}$ for general $k, l, m$ to describing $\SSS^{m,1}_{m}$ for all $m$. In this section we use Lemma~\ref{lemma: stable coend} to show that $\SSS^{m,1}_{m}$ is equivalent to the suspension spectrum of the complex of direct-sum decompositions of $\CC^m$, which we denote $\cL^\diamond_m$. This leads to a complete description of $\SSS^{k,l}_{m}$ in terms of the complexes $\cL^\diamond_m$ (Theorem~\ref{theorem: main}).

The complexes $\cL^\diamond_m$ were first introduced in~\cite{Ar}, and were studied in detail in~\cite{Banff} and~\cite{AL}. They play a role in orthogonal calculus, and also in describing the subquotients of the rank filtration of $K$-theory~\cite{AL-Crelle, AL-Fundamenta}. They have some remarkable homotopical properties, that we will recall in the next section (Proposition~\ref{proposition: L_m facts}).

Our proof that $\SSS^{m,1}_{m}$ is equivalent to the suspension spectrum of $\cL^\diamond$ goes though an intermediate complex, which we call the complex of {\it ordered} direct-sum decompositions. Let us give the formal definition.
\begin{define}
Let $\cD^{\mathrm{o}}_m$ be the following category objects in topological spaces. Its objects are ordered tuples $(E_1, \ldots, E_t)$ of pairwise orthogonal proper, non-trivial vector subspaces of $\CC^m$, whose direct sum is $\CC^m$. A morphism $(E_1, \ldots, E_t)\to (F_1, \ldots, F_s)$ consists of a surjective function $\alpha\colon \{1,\ldots, t\}\twoheadrightarrow \{1,\ldots, s\}$ such that for each $1\le i\le t$, $E_i\subseteq F_{\alpha(i)}$.
\end{define}
We call the category $\cD^{\mathrm{o}}_m$ the category of proper, {\it ordered} direct-sum decompositions of $\CC^m$. The set of objects and the set of morphisms both have a topology. There is a natural action of $U(m)$ on $\cD^{\mathrm{o}}_m$, and object and morphism sets of are topologized as unions of $U(m)$-orbits.

There is a convenient presentation of $\cD^{\mathrm{o}}_m$ as the Grothendieck construction applied to the functor $\Gred_{m,1}^m$ of Definition~\ref{definition: Gred}. Let us recall the definition of (a version of) the Grothendieck construction.
\begin{define}\label{def: grothendieck}
Suppose $\cC$ is a small category and $F\colon \cC \to \uTop$ is a functor. The Grothendieck construction on $F$ (a.k.a the wreath product of $\cC$ and $F$) is the following pointed topological category, denoted $\cC\wr F$. The objects of $\cC\wr F$ are pairs $(c, x)$ where $c$ is an object of $\cC$, and $x\in F(c)$. A morphism $(c, x)\to (d, y)$ in $\cC\wr F$ is a morphism $\alpha\colon c \to d$ in $\cC$ such that $F(\alpha)(x)=y$. The space of objects of $\cC\wr F$ is topologized as the disjoint union $\coprod_c F(c)$ indexed by the objects of $\cC$, and the space of morphisms is topologized as the disjoint union $\coprod_{c\to d}F(c)$, indexed by morphisms of $\cC$.
\end{define}
The following well-known lemma can be thought of as a topological analogue of Thomason's homotopy colimit theorem.
\begin{lem}\label{lemma: topological thomason theorem}
Suppose $\cC$ is a small category and $F\colon \cC \to \uTop$ is a functor. There is a natural equivalence
\[
\hocolim_{\cC}F\simeq |\cC\wr F|.
\]
\end{lem}
\begin{proof}
It is easy to see that the simplicial nerve of $\cC\wr F$ is {\it isomorphic}, as a simplicial space, to Bousfield and Kan's simplicial model for $\hocolim_{\cC}F$. In fact, both simplicial spaces are given in simplicial degree $k$ by the space
\[
\coprod_{c_0\to\cdots\to c_k} F(c_0).
\]
The $i$-th face map $d_i$ is defined by dropping $c_i$ and, if $i=0$, using the functoriality of $F$ to map $F(c_0)$ to $F(c_1)$. The degeneracy map $s_i$ is defined by duplicating $c_i$.
\end{proof}
Now recall that we have a functor $\Gred_{m,1}^m\colon \Epi\to \uTop$ (Definition~\ref{definition: Gred}). Let $\Epi^{> 1}$ be the full subcategory of $\Epi$ consisting of sets of cardinality greater than $1$. By slight abuse of notation we denote the restriction of $\Gred_{m,1}^m$ to $\Epi^{>1}$ also by $\Gred_{m,1}^m$.
The  following lemma is straightforward from the definitions:
\begin{lem}\label{lemma: Grothendieck}
There is an isomorphism of topological categories \[\Epi^{>1}\wr \Gred_{m,1}^m\cong {\cD^{\mathrm{o}}_m}.\]
\end{lem}
Given a space $X$, let $X^\diamond$ denote the unreduced suspension of $X$. We have the following connection between $\SSS^{m,1}_{m}$ and $\cD^{\mathrm{o}}_m$.
\begin{prop}\label{prop: ordered}
There is a natural equivalence
\[
\SSS^{m,1}_{m}\simeq \Sigma^\infty |\cD^{\mathrm{o}}_m|^\diamond.
\]
\end{prop}
\begin{proof}
We saw in lemma~\ref{lemma: stable coend} that
\[
\SSS^{m,1}_{m}\simeq \Sigma^\infty \int^{\Epi}_{\mathrm h}\mathtt{I}\wedge  {\Gred_{m,1}^m}_+
\]
where $\mathtt{I}\colon\Epi^{\op}\to \Top$ is the functor that sends $\underline{1}$ to $ S^0$ and sends all other objects to $*$. So we need to show that there is an equivalence of pointed spaces
\[
\int^{\Epi}_{\mathrm h}\mathtt{I}\wedge  {\Gred_{m,1}^m}_+ \simeq |\cD^{\mathrm{o}}_m|^\diamond.
\]

Let $\mathtt{S}\colon \Epi^{\op}\to \Top$ be the constant functor $\mathtt{S}(\underline{t})\equiv S^0$. Let $\mathtt{S}^{>1}\colon \Epi^{\op}\to \Top$ be the functor $\mathtt{S}^{>1}(\underline{1})=*$ and $\mathtt{S}^{>1}(\underline{t})\equiv  S^0$ for $t>1$. Then there is a homotopy cofibration sequence of functors $\mathtt{S}^{>1}\to \mathtt{S}\to \mathtt{I}$. It follows that there is a homotopy cofibration sequence of coends
\[
 \int^{\Epi}_{\mathrm h}\mathtt{S}^{>1}\wedge  {\Gred_{m,1}^m}_+ \to  \int^{\Epi}_{\mathrm h}\mathtt{S}\wedge  {\Gred_{m,1}^m}_+ \to  \int^{\Epi}_{\mathrm h}\mathtt{I}\wedge {\Gred_{m,1}^m}_+
\]
It is a standard fact that $$ \int^{\Epi}_{\mathrm h}\mathtt{S}\wedge  {\Gred_{m,1}^m} \simeq {\hocolim_{\Epi}}^*  ({\Gred_{m,1}^m}_+)\cong (\hocolim_{\Epi}  {\Gred_{m,1}^m})_+$$
(here $\hocolim^*$ denotes pointed homotopy colimit, while $\hocolim$ denotes unpointed homotopy colimit).
Since $\Epi$ has a final object $\underline{1}$, it follows that $${\hocolim_{\Epi}}^*  {\Gred_{m,1}^m}_+\simeq {\Gred_{m,1}^m}_+ (\underline{1}) = S^0.$$
On the other hand, since $\underline{1}$ is the initial object of $\Epi^{\op}$, and $\mathtt{S}^{>1}(\underline{1})=*$ it follows easily that $\mathtt{S}^{>1}$ is equivalent to the functor obtained by restricting $\mathtt{S}$ to the subcategory ${\Epi^{>1}}^{\op}$ of sets of cardinality greater than $1$, and then taking derived left Kan extension back to ${\Epi}^{\op}$.
By standard adjunctions, it follows that there are equivalences
\[
 \int^{\Epi}_{\mathrm h}\mathtt{S}^{>1}\wedge  {\Gred_{m,1}^m}_+ \simeq  \int^{\Epi^{>1}}_{\mathrm h}\mathtt{S}\wedge  {\Gred_{m,1}^m}_+\simeq (\hocolim_{\Epi^{>1}} {\Gred_{m,1}^m})_+
\]
It follows that there is a homotopy cofibration sequence
\[
(\hocolim_{\Epi^{>1}} {{\Gred_{m,1}^m}})_+ \to  S^0 \to  \int^{\underline{t}\in\Epi}_{\mathrm h}\mathtt{I}\wedge {\Gred_{m,1}^m}_+
\]
By lemma~\ref{lemma: Grothendieck} ${\cD^{\mathrm{o}}_m}$ is the Grothendieck construction on $\Gred_{m,1}^m$. It follows by Lemma~\ref{lemma: topological thomason theorem} that
\[
\hocolim_{\Epi^{>1}} {\Gred_{m,1}^m} \simeq |\cD^{\mathrm{o}}_m|.
\]
So we have a homotopy cofibration sequence
\[
|\cD^{\mathrm{o}}_m|_+ \to S^0 \to \int^{\underline{t}\in\Epi}_{\mathrm h}\mathtt{I}\wedge {\Gred_{m,1}^m}_+
\]
This implies that $\int^{\underline{t}\in\Epi}_{\mathrm h}\mathtt{I}\wedge {\Gred_{m,1}^m}_+\simeq |\cD^{\mathrm{o}}_m|^\diamond$.
\end{proof}
Our next step is to show that $\cD^{\mathrm o}_m$ can be replaced with a smaller category, which we call the poset of {\it unordered} direct-sum decompositions. First, the definition.

\begin{define}\label{definition: decompositions}
Let $\cD_m$ be the following category objects in topological spaces. Its objects are  {\em unordered} sets $\{E_i\mid i\in I\}$ of pairwise orthogonal proper, non-trivial vector subspaces of $\CC^m$, whose direct sum is $\CC^m$. There is a unique morphism $\{E_i\mid i\in I\}\to \{F_j\mid j\in J\}$ if for each $i\in I$ there is a (necessarily unique) $j\in J$ such that $E_i\subseteq F_j$. In keeping with recent literature, the geometric realization of $\cD_m$ will be denoted $\cL_m$, and its unreduced suspension is therefore $\cL_m^\diamond$.
\end{define}

As with $\cD^{\mathrm{o}}_m$, there is a natural action of $U(m)$ on $\cD_m$ and both the sets of objects and morphisms of $\cD_m$ are topologized as unions of $U(m)$-orbits.
We note that for any two objects $P, Q$ of $\cD^{\mathrm{o}}_m$ there is at most one morphism from $P$ to $Q$. In other words, $\cD^{\mathrm{o}}_m$ is {\it a topological preorder}. By contrast, the category $\cD_m$ is a topological poset: it is the poset of isomorphism classes of $\cD^{\mathrm{o}}_m$. The category $\cD_m$ will be referred to as the category, or poset, of proper, {\em unordered} direct-sum decompositions of $\CC^m$.

There is a topological functor $q\colon \cD^{\mathrm{o}}_m\to \cD_m$, which forgets the order of the components.
\begin{prop}\label{proposition: ordered to unordered}
The natural functor $q\colon \cD^{\mathrm{o}}_m\to \cD_m$ induces an equivalence of geometric realizations $|\cD^{\mathrm{o}}_m|\xrightarrow{\simeq} |\cD_m|$
\end{prop}
\begin{proof}
We are going to use Quillen's theorem A. We need a version of it that is valid for topological categories. There are several such versions scattered in the literature, we will use~\cite[Theorem 4.7]{EbRW}. According to this theorem, it is enough if we prove the following
\begin{enumerate}
\item \label{contractible}For every object $\Lambda$ of $\cD_m$, the classifying space of the over category $q/\Lambda$ is contractible.
\item \label{fibrant} The map from the morphism space of $\cD^{\mathrm{o}}_m$ to the object space of $\cD^{\mathrm{o}}_m$, that sends every morphism to its target, is a fibration (in the language of~\cite{EbRW}, this means that $\cD^{\mathrm{o}}_m$ is right fibrant).
\item \label{fibration} The map from the space of objects of the over category $q/\cD_m$ to the space of objects of $\cD_m$, that sends a morphism to its target, is a fibration. Here $q/\cD_m$ is the category of arrows in $\cD_m$ of the form $q(\Theta)\to \Lambda$, where $\Theta$ is an object of $\cD^{\mathrm{o}}_m$.
\end{enumerate}
For part~\eqref{contractible}, let $\Lambda=\{F_i \mid i\in I\}$ be an object of $\cD_m$, i.e., an unordered collection of pairwise orthogonal non-trivial subspaces of $\CC^m$ whose direct sum is $\CC^m$. Let $t$ be the number of elements of $I$ and choose a bijection $I\cong \{1, \ldots, t\}$. Then $\widetilde \Lambda=(F_1, \ldots, F_t)$ is a choice of lift of $\Lambda$ to an object of $\cD^{\mathrm{o}}_m$. An object of $q/\Lambda$ consists of an object $\Theta=(E_1, \ldots, E_s)$ such that each $E_i$ is a subspace of $F_j$ for some (necessarily unique) $j$. It follows that there exists a {\em unique} surjection $\alpha\colon \{1, \ldots, s\}\twoheadrightarrow \{1, \ldots, t\}$ such that $E_i\subset F_{\alpha(i)}$ for all $i$. This means that there is a unique morphism from $\Theta$ to $\widetilde \Lambda$ in $q/\Lambda$. Thus the category $q/\Lambda$ has a (not necessarily unique) terminal object, and therefore its classifying space is contractible.

For part~\eqref{fibrant}, using the identification of $\cD^{\mathrm{o}}_m$ with the Grothendieck construction $\Epi^{>1}\wr \Gred_{m,1}^m$ (Lemma~\ref{lemma: Grothendieck}), the map from the space of morphisms of $\cD^{\mathrm{o}}_m$ to the space of objects of $\cD^{\mathrm{o}}_m$ which sends each morphism to its target, has the following form
\begin{equation}\label{eq: target}
\coprod_{\underline{s}\twoheadrightarrow \underline{t}\in \Epi^{>1}} \coprod_{\underset{ m_i>0, \Sigma m_i = m\}}{\{(m_1, \ldots, m_s)\mid}} U(m)/\prod_{i=1}^s U(m_i) \to \coprod_{\underline{t}\in \Epi^{>1}} \coprod_{\underset{ n_i>0, \Sigma n_j = m\}}{\{(n_1, \ldots, n_t)\mid}} U(m)/\prod_{j=1}^t U(n_j)
\end{equation}
where for every surjective function $\underline{s}\twoheadrightarrow \underline{t}$, the space $U(m)/\prod_{i=1}^s U(m_i)$ is sent to $U(m)/\prod_{j=1}^s U(n_j)$, where for each $j=1, \ldots, t$, $n_j=\Sigma_{i\in \alpha^{-1}(j)} m_i$, by the canonical quotient map associated with the sub-conjugation of $\prod_{i=1}^s U(m_i)$ into $\prod_{j=1}^s U(n_j)$ induced by $\alpha$. The map~\eqref{eq: target} is clearly a fibration. Indeed, it is a $U(m)$ equivariant map between disjoint union of $U(m)$-orbits, and such a map is necessarily a fibration.

Finally, the proof of part~\eqref{fibration} is similar to that of part~\eqref{fibrant}. Since $\cD_m$ is the poset of isomorphism classes of the pre-order $\cD^{\mathrm{o}}_m$, the space of objects of the category $q/\cD_m$ is the quotient of the space of morphisms of $\cD^{\mathrm{o}}_m$ by the action of the groupoid of isomorphisms of the target. Similarly, the space of objects of $\cD_m$ is the quotient of the space of objects of $\cD^{\mathrm{o}}_m$ by the groupoid of isomorphisms. This means that we have the following map
\[
\left(\coprod_{\underline{s}\twoheadrightarrow \underline{t}} \coprod_{\underset{ m_i>0, \Sigma m_i = m\}}{\{(m_1, \ldots, m_s)\mid}} U(m)/\prod_{i=1}^s U(m_i)\right)_{\operatorname{Iso}(t)} \to \left(\coprod_{\underline{t}} \coprod_{\underset{ n_i>0, \Sigma n_j = m\}}{\{(n_1, \ldots, n_t)\mid}} U(m)/\prod_{j=1}^t U(n_j)\right)_{\operatorname{Iso}(t)}
\]
The action of the groupoid of isomorphisms of the variable $\underline{t}$ respects the action of $U(m)$. It follows that the resulting map is still a $U(m)$-equivariant map between disjoint union of orbits, and therefore is a fibration.
\end{proof}

Propositions~\ref{prop: ordered} and~\ref{proposition: ordered to unordered}, together with equation~\eqref{equation: stable reduction} give us an equivalence
\begin{equation}\label{equation: Lm formula for Skl}
\SSS^{k,l}_m\simeq U(k)/U(k-lm)_+ \wedge_{U(m)} \Sigma^\infty\cL_m^\diamond.
\end{equation}
We also want to describe the composition morphisms $\mathbb{S}^{k,l}_m\wedge \mathbb{S}^{j,k}_n\to \mathbb{S}^{j,l}_{mn}$. We begin by observing that tensor product induces a natural map $\cL^\diamond_m \wedge \cL^\diamond_n\to \cL^\diamond_{mn}$ as follows. Suppose that $\cE=\{E_i\mid i\in I\}$ and $\cF=\{F_j\mid j\in J\}$ are direct-sum decompositions of $\CC^m$ and $\CC^n$ respectively. Then $\cE\otimes \cF:=\{E_i\otimes F_j\mid (i,j)\in I\times J\}$ is a direct-sum decomposition of $\CC^m\otimes \CC^n \cong \CC^{mn}$. Note that if at least one of $\cE$, $\cF$ is a proper decomposition (i.e., has more than one component) then $\cE\otimes\cF$ is a proper decomposition as well. This means that the tensor product induces a map $\cL^\diamond_m \wedge \cL^\diamond_n\to \cL^\diamond_{mn}$ as desired. Note that this map is equivariant with respect to the tensor product homomorphisms $U(m)\times U(n)\to U(mn)$.

Next, we extend it to a map
\begin{multline*}
\left(U(k)/U(k-lm)_+ \wedge_{U(m)} \cL_m^\diamond\right)\wedge \left(U(j)/U(j-kn)_+ \wedge_{U(n)} \cL_n^\diamond\right)\to \\ \to U(j)/U(j-lnm)_+\wedge_{U(nm)} \cL^\diamond_{nm}.
\end{multline*}
Now recall that $U(k)/U(k-lm)\cong \Inj(\CC^{lm}, \CC^k)$ and $U(j)/U(j-kn)\cong \Inj(\CC^{kn}, \CC^j)$. Given homomorphisms  $g\in \Inj(\CC^{lm}, \CC^k)$ and $f\in  \Inj(\CC^{kn}, \CC^j)$, we may form the homomorphism $f\circ g^{n}\colon {\CC^{lmn}}\hookrightarrow \CC^j$. Clearly, this defines a map $U(k)/U(k-lm)\times U(j)/U(j-kn)\to U(j)/U(j-lnm)$. This map is equivariant with respect to the tensor product homomorphism $U(m)\times U(n)\to U(mn)$. Combining it with the map  $\cL^\diamond_m \wedge \cL^\diamond_n\to \cL^\diamond_{mn}$ defined earlier, we obtain the desired map.

We are ready to state the main theorem of the paper.
\begin{thm}\label{theorem: main}
There is an equivalence
\[
\SSS^{k,l}_m\simeq \Sigma^\infty U(k)/U(k-lm)_+ \wedge_{U(m)} \cL_m^\diamond\cong \Sigma^\infty \Inj(\CC^{ml}, \CC^k)_+\wedge_{U(m)} \cL_m^\diamond.
\]
Under this equivalence, the composition product $\SSS^{k,l}_m\wedge \SSS^{j,k}_n \to \SSS^{j,l}_{mn}$ corresponds to the map
\begin{multline*}
\left(\Inj\left(\CC^{ml}, \CC^k\right)_+ \wedge_{U(m)} \cL_m^\diamond\right)\wedge \left(\Inj\left(\CC^{nk}, \CC^j\right)_+ \wedge_{U(n)} \cL_n^\diamond\right)\to \\ \to \Inj\left(\CC^{nml}, \CC^j\right)_+\wedge_{U(nm)} \cL^\diamond_{nm}.
\end{multline*}
that were defined above.
\end{thm}
\begin{proof}
We already proved the formula for $\SSS^{k,l}_{m}$ (equation~\ref{equation: Lm formula for Skl}). It remains to check the statement about the composition product. Recall that $\SSS^{k,l}_m$ is the stabilization of the functor $G_{k,l}^{m}$. The composition product is determined by the natural transformation $\Gred_{k,l}^m(v)\wedge \Gred_{j,k}^n(u)\to \Gred_{j,l}^{mn}(u\times v)$. An analysis of this composition map shows that it is induced by disjoint union of maps of the form
\begin{multline*}
\Inj(\CC^{(m_1+\cdots+m_t)l}, \CC^k)/_{\prod_{j=1}^t U(m_j)}\times \Inj(\CC^{(n_1+\cdots+n_s)k}, \CC^j)/_{\prod_{i=1}^s U(n_i)} \to \\ \to {\Inj(\CC^{(\sum_{i=1, j=1}^{i=s, j=t} n_i m_j)l}, \CC^j)/_{\prod_{i=1, j=1}^{i=s, j=t} U(n_im_j)}}
\end{multline*}
that sends $(q,p)$ to $p\circ q^{n}$, where $n=n_1+\cdots+n_s$. Note that the decomposition of $\CC^{mn}$ associated with the target of this map is the tensor product of the given decompositions of $\CC^m$ and $\CC^n$, just as was claimed. This induces the claimed map of spectra.
\end{proof}

\section{Some calculations of $\SSS^{k,l}$} \label{ss: calculations}
In this section we calculate the spectra $\SSS^{k,l}$ in some cases, and also prove that the map $\SSS^{k,l}\to ku$ is an isomorphism on $\pi_0$. Our main tool is Theorem~\ref{theorem: main}, which expresses the subquotients of the rank filtration in terms of the complexes $\cL_m^\diamond$. To use it, we need to know something about the complexes $\cL_m^\diamond$. So let us begin by reviewing some of the rather remarkable properties of these complexes that were uncovered in~\cite{Ar, AL-Crelle, Banff, AL}. The following proposition lists the relevant facts.
\begin{prop}\label{proposition: L_m facts}
\begin{enumerate}
\item The space $\cL_m^\diamond$ is rationally contractible for $m>1$. \label{rational}
\item The space $\cL_m^\diamond$ is (integrally) contractible unless $m$ is a prime power. \label{integral}
\item If $m=p^k$ with p a prime and $k>0$, then $\cL_{p^k}^\diamond$ is $p$-local, and has chromatic type $k$.\label{p local}
\end{enumerate}
\end{prop}
\begin{proof}
Except for the statement about the chromatic type, this is~\cite[proposition 9.6]{AL-Crelle}, which in turn relies on~\cite{Ar}. The statement about the chromatic type is part of~\cite[Theorem 2.2]{Ar2}.

The proofs in~\cite{Ar2} are based on a rather deep connection between $\cL_m^\diamond$ and the calculus of functors. Since part~\eqref{rational} plays a prominent role in our applications, we indicate an independent, more direct way to prove this part.
The space $\cL_m^\diamond$ is equivalent to the total homotopy cofiber of the following $m-1$-dimensional cubical diagram.  Suppose $U=\{i_1, \ldots, i_k\}\subseteq \{2, \ldots, m\}$, with $i_1>\ldots >i_k$. Let $\Chi(U)$ be the space of chains of decompositions of $\CC^m$ of the form $(\Lambda_1<\cdots<\Lambda_k)$ where each $\Lambda_j$ has $i_j$-components. If $U$ is empty then $\Chi(U)=*$.  Note that in general $\Chi(U)$ is a disjoint union of $U(m)$-orbits. The assignment $U\mapsto \Chi(U)$ defines a diagram indexed on the opposite of the poset of subsets of $\{2, \ldots, m\}$, i.e., an $m-1$-dimensional cubical diagram. It is elementary to show that $\cL_m^\diamond$ is equivalent to the total homotopy cofiber of the cube $\Chi$. For example, in the case $m=3$, $\Chi$ is the following square of $U(3)$-orbits.
\begin{equation}\label{eq: L3}
\begin{array}{ccc}
U(3)/\Sigma_2\wr U(1)\times U(1) & \to & U(3)/U(2)\times U(1) \\
\downarrow & & \downarrow \\
U(3)/\Sigma_3\wr U(1) & \to & U(3)/U(3) \end{array}
\end{equation}
Here the upper right corner is $\Chi(\{2\})$, the space of decompositions of $\CC^3$ with $2$ components, the lower left corner is $\Chi(\{3\})$, the space of decompositions with $3$ components, and the upper left corner is $\Chi(\{2,3\})$, the space of morphisms from a decomposition with $3$ components to a decomposition with $2$ components.

Each one of the horizontal maps in~\eqref{eq: L3} is a map of $U(3)$-orbits, induced by subgroup inclusions $\Sigma_2\wr U(1)\times U(1) \to U(2)\times U(1)$ and $\Sigma_3\wr U(1) \to U(3)$. Note that in both of these cases, the subgroup that is being included is the normalizer of a maximal torus. It follows that each one of the horizontal maps is a rational equivalence, and therefore the total cofiber of~\eqref{eq: L3} is trivial in rational homology. Since it is simply connected, it is also trivial in rational homotopy.

More generally suppose $m>i_1$ and consider the map $\Chi(\{m, i_1, \ldots, i_k\})\to \Chi(\{i_1, \ldots, i_k\})$. This map is a disjoint union of maps between $U(m)$-orbits. For each path component of $\Chi(\{m, i_1, \ldots, i_k\})$, the isotropy group is the normalizer of a maximal torus of the isotropy group of a corresponding component of $\Chi(\{i_1, \ldots, i_k\})$. It follows that the map $\Chi(\{m, i_1, \ldots, i_k\})\to \Chi(\{i_1, \ldots, i_k\})$ is always a rational equivalence, and therefore the total homotopy cofiber of $\Chi$, which is $\cL_m^\diamond$, is rationally trivial.
\end{proof}
\begin{cor}\label{cor:rat}
The map
\[
\SSS^{k,l,1}\to \SSS^{k,l}.
\]
is a rational equivalence.
\end{cor}
\begin{proof}
To see this, consider the filtration
\[
*=\SSS^{k,l,0}\to \SSS^{k,l,1}\to \SSS^{k,l,2}\to \cdots \to \SSS^{k,l,\lfloor\frac{k}{l}\rfloor}=\SSS^{k,l}
\]
It follows from Theorem \ref{theorem: main} and part~\ref{rational} of Proposition~\ref{proposition: L_m facts} that for all $m>1$ the homotopy cofiber $\SSS^{k,l}_m$ of the map $\SSS^{k,l,m-1}\to \SSS^{k,l,m}$ is rationally trivial. It follows that the map $\SSS^{k,l,1}\to \SSS^{k,l}$ is a rational equivalence.
\end{proof}

Here is an explicit description of $\cL_m^\diamond$ for some values of $m$.
\begin{prop}
\begin{enumerate}
\item $\cL_1^\diamond\cong S^0$
\item $\cL_2^\diamond\cong \Sigma\RR P^2$.
\item More generally, if $p$ is a prime then $\cL_p^\diamond$ is a union of $p-1$ shifted copies of the mod $p$ Moore space.
\end{enumerate}
\end{prop}

As a first application of Theorem~\ref{theorem: main} and Proposition~\ref{proposition: L_m facts}, let us prove that the map $\SSS^{k,l}\to ku$ induces an isomorphism on $\pi_0$.
\begin{lem}\label{lem: over ku}
Assume that $k\ge l$. The map $\SSS^{k, l}\to \SSS^{\infty , l}\simeq ku$ induces an isomorphism on $\pi_0$.
\end{lem}
\begin{rem}
We remind the reader that if $k<l$, $\SSS^{k,l}\simeq *$.
\end{rem}
\begin{proof}
We saw in Section~\ref{section: stable rank} that the mapping spectra $\SSS^{k,l}$ are filtered by a sequence of spectra
\[
*=\SSS^{k,l,0}\to \SSS^{k,l,1}\to \SSS^{k,l,2}\to \cdots \to \SSS^{k,l,\lfloor\frac{k}{l}\rfloor}=\SSS^{k,l}
\]
Consider the commutative diagram
\[
\begin{array}{ccc}
 \SSS^{k,l,1} & \to & \SSS^{k,l}\\
 \downarrow & & \downarrow \\
 \SSS^{\infty, l, 1} & \to & \SSS^{\infty , l}
 \end{array}
\]
We will prove that the left, top and bottom maps in this diagram induce an isomorphism on $\pi_0$. It then follows that the right map induces an isomorphism on $\pi_0$, which is what we want to prove.

By Theorem~\ref{theorem: main},  $$\SSS^{k,l,1}\simeq\SSS^{k,l}_1\simeq \Sigma^\infty \Inj(\CC^l, \CC^k)/U(1)_+,$$ and similarly $$\SSS^{\infty,l,1}\simeq \Sigma^\infty \Inj(\CC^l, \CC^\infty)/{U(1)}_+.$$ Since $\Inj(\CC^l, \CC^k)/U(1)$ and $\Inj(\CC^l, \CC^\infty)/{U(1)}$ are path-connected spaces, the map $\SSS^{k,l,1}\to \SSS^{\infty,l,1}$ induces on $\pi_0$ the isomorphism from $\ZZ$ to itself.

To analyze the map $ \SSS^{k,l,1}  \to  \SSS^{k,l}$ recall, again from Theorem~\ref{theorem: main}, that the subquotient $\SSS^{k,l,m}/\SSS^{k,l,m-1}$ is equivalent to the suspension spectrum of $\Inj(\CC^{ml}, \CC^k)_+ \wedge_{U(m)} \cL_m^\diamond$. For $m>1$ the space $\cL_m$ is path-connected, so $\cL_m^\diamond$ is simply-connected. It follows that $\SSS^{k,l,m}/\SSS^{k,l,m-1}$ is $1$-connected for $m>1$, and therefore the map $\SSS^{k,l,1}\to \SSS^{k,l}$ is $1$-connected, and in particular it induces an isomorphism on $\pi_0$. The same argument applies in the case $k=\infty$, which completes the proof.
\end{proof}
Since the rank filtration of $\SSS^{k,l}$ has length $\lfloor\frac{k}{l}\rfloor$, we can conclude that if $k<2l$ then $\SSS^{k,l,1}$ is in fact equivalent to $\SSS^{k,l}$.
\begin{lem}
If $l\le k \le 2l-1$ then $\SSS^{k,l}$ is integrally equivalent to $$\Sigma^\infty U(k)/U(k-l)\times U(1)_+.$$ In particular, for $k=l$ the spectrum $\SSS^{k,k}=\operatorname{End}_{\NSp}(\Sigma^\infty_{\NC} M_k)$ is equivalent to $\Sigma^\infty PU(k)_+$: the group ring spectrum of the projective unitary group.
\end{lem}

\section{On the rationalization and $p$-localization of $\cM$}\label{section: localizations}
Let  $\cC$ be a stable presentable closed symmetric monoidal $\infty$-category. Denote by $\otimes$ the tensor product, by $1_\cC$ the unit and by $\underline{\Hom}(\bullet,\bullet)$ the internal hom.
For every $n\in \mathbb{Z}$ and an object $X \in \cC$ there is a natural multiplication by $n$ map $[n]\colon X \to X$. We will say that an object $X \in \cC$ is \emph{rational}
if for every $n\neq 0$ the map $[n]\colon X \to X$ is an isomorphism. Simlarly, we will say that $X$ is \emph{$p$-local} for a prime $p$, if $[n]\colon X \to X$ is an isomorphism for every $n$ that is not divisible by $p$. We denote the collection of rational objects in $\cC$ by $\cC_{\mathbb{Q}}$ and the collection of $p$-local objects by $\cC_{(p)}$. If $\cC=\cC_{\QQ}$ (resp. $\cC=\cC_{(p)}$) then we say that $\cC$ is rational (resp. $p$-local). The naturality of $[n]$ implies that $\cC_{\mathbb{Q}}$ and $\cC_{(p)}$ are closed in $\cC$ under all small limits and colimits and that for every $X\in \cC$ and $Y\in \cC_{\mathbb{Q}}$ (resp. $Y\in \cC_{(p)}$) we have $X\otimes Y, \underline{\Hom}(X,Y) \in \cC_{\mathbb{Q}}$ (resp. $X\otimes Y, \underline{\Hom}(X,Y) \in \cC_{(p)}$). We thus get that $\cC_{\mathbb{Q}}$ and $\cC_{(p)}$ are themselves stable presentable closed symmetric monoidal $\infty$-categories.
Further the inclusion $$i^{\cC}_\mathbb{Q}:\cC_\mathbb{Q} \subset \cC$$ admits a symmetric monoidal left adjoint called \emph{rationalization}
$$L^{\cC}_{\mathbb{Q}}\colon \cC \to \cC_{\mathbb{Q}}.$$
Same holds for $p$-localization.

Further, the left adjoints are given by the following formulas
$$L_{\mathbb{Q}}(X) = L_{\mathbb{Q}}(1_\cC)\otimes X = \colim \left[  X \xrightarrow{[1]}X \xrightarrow{[2]}X \xrightarrow{[3]}X \cdots \right]$$
$$L_{(p)}(X) = L_{{p}}(1_\cC)\otimes X = \colim \left[  X \xrightarrow{[p'_1]}X \xrightarrow{[p'_2]}X \xrightarrow{[p'_3]}X \cdots \right]$$
where $p'_1, p'_2, \ldots$ is the list of integers not divisible by $p$.

Since $\NSp$ is left-tensored over $\Sp$, we have that $\NSp_{\mathbb{Q}}$ (resp. $\NSp_{(p)}$) is left-tensored over $\Sp_{\mathbb{Q}}$ (resp. $\Sp_{(p)}$).
Let ${\cM^{\mathbb{Q}}}$ (resp. ${\cM^{(p)}}$) to be the full $\Sp_{\mathbb{Q}}$-enriched (resp. $\Sp_{(p)}$-enriched) subcategory of $\NSp_{\mathbb{Q}}$ (resp. $\NSp_{(p)}$) spanned by $$L^{\NSp}_{\mathbb{Q}}(\Sigma^\infty_{\NC} M_n) \quad \left(\mbox{ resp. } L^{\NSp}_{(p)}(\Sigma^\infty_{\NC} M_n)\right)$$ for  $n \in \mathbb{N}$.
\begin{lem}\label{l:M^Q}
For all $k,l\in \mathbb{N}$ there are equivalences
\[
\Hom_{{\cM^{\mathbb{Q}}}}(L^{\NSp}_{\mathbb{Q}}(\Sigma^\infty_{\NC} M_k),L^{\NSp}_{\mathbb{Q}}(\Sigma^\infty_{\NC} M_l)) \simeq L^{\Sp}_{\mathbb{Q}}\Hom_{\NSp} ( \Sigma^\infty_{\NC} M_k, \Sigma^\infty_{\NC} M_l)
\]
and
\[
\Hom_{{\cM^{(p)}}}(L^{\NSp}_{(p)}(\Sigma^\infty_{\NC} M_k),L^{\NSp}_{(p)}(\Sigma^\infty_{\NC} M_l)) \simeq L^{\Sp}_{(p)}\Hom_{\NSp} ( \Sigma^\infty_{\NC} M_k, \Sigma^\infty_{\NC} M_l)
\]
\end{lem}
\begin{proof}
We will go over the (very straightforward) proof of the rational case. The proof of the $p$-local case is practically identical.
\begin{multline*}
\Hom_{{\cM^{\mathbb{Q}}}}(L^{\NSp}_{\mathbb{Q}}(\Sigma^\infty_{\NC} M_k),L^{\NSp}_{\mathbb{Q}}(\Sigma^\infty_{\NC} M_l))
=\Hom_{\NSp_{\mathbb{Q}}}(L^{\NSp}_{\mathbb{Q}}(\Sigma^\infty_{\NC} M_k),L^{\NSp}_{\mathbb{Q}}(\Sigma^\infty_{\NC} M_l)) = \\ =\Hom_{\NSp}(\Sigma^\infty_{\NC} M_k,L^{\NSp}_{\mathbb{Q}}(\Sigma^\infty_{\NC} M_l))= \\=\Hom_{\NSp}(\Sigma^\infty_{\NC} M_k,\colim \left[  \Sigma^\infty_{\NC} M_l \xrightarrow{[1]}\Sigma^\infty_{\NC} M_l \xrightarrow{[2]}  \Sigma^\infty_{\NC} M_l \xrightarrow{[3]} \cdots \right] ) =  \\ =\colim \left[  \Hom_{\NSp}(\Sigma^\infty_{\NC} M_k,\Sigma^\infty_{\NC} M_l)\xrightarrow{[1]}\Hom_{\NSp}(\Sigma^\infty_{\NC} M_k,\Sigma^\infty_{\NC} M_l)\xrightarrow{[2]} \right. \\ \left. \xrightarrow{[2]} \Hom_{\NSp}(\Sigma^\infty_{\NC} M_k,\Sigma^\infty_{\NC} M_l )\xrightarrow{[3]} \cdots \right]  = L^{\Sp}_{\mathbb{Q}}\Hom_{\NSp} ( \Sigma^\infty_{\NC} M_k, \Sigma^\infty_{\NC} M_l).
\end{multline*}
Here the first equality is by definition, the second equality is using the adjunction $L^{\NSp}_{\mathbb{Q}}\vdash i^{\NSp}_{\mathbb{Q}}$, the third equality is the formula for $L^{\NSp}_{\mathbb{Q}}$, the forth equality is by the compactness of $\Sigma^\infty_{\NC} M_k$ and the fifth equality uses the formula for $L^{\Sp}_{\mathbb{Q}}$.
\end{proof}
The main theorems of~\cite{ABS1} have rational and $p$-local analogs with completely analogous proofs.
\begin{thm}\label{t:modules monoida-rl}
Let $\cD$ be a symmetric monoidal cocomplete rational (resp. $p$-local) $\infty$-category. Suppose that there is a small set $C$ of compact objects in $\cD$, that generates $\cD$ under colimits and desuspentions. Assume that $1_\cD \in C$ and $C$ is closed under tensor product.  Thinking of $\cD$ as left-tensored over $\Sp_{\mathbb{Q}}$ (resp. $\Sp_{(p)}$), we let $\cC$ be the full $\Sp_{\mathbb{Q}}$-enriched (resp. $\Sp_{(p)}$-enriched) subcategory of $\cD$ spanned by $C$. Then we have a natural symmetric monoidal  functor of categories left-tensored over $\Sp_{\mathbb{Q}}$ (resp. $\Sp_{(p)}$)
  $$P_{\Sp_{\mathbb{Q}}}(\cC)\xrightarrow{\sim}\cD \quad \left(\mbox{resp. } P_{\Sp_{(p)}}(\cC)\xrightarrow{\sim}\cD\right),$$
  which is an equivalence of the underlying $\infty$-categories and sends each representable presheaf $Y(c)$ to $c\in C$.
\end{thm}
\begin{thm}\label{theorem: rational main presentation}
The $\Sp_{\mathbb{Q}}$-enriched (resp. $\Sp_{(p)}$-enriched) category ${\cM^{\mathbb{Q}}}$ (resp. ${\cM^{(p)}}$) acquires a canonical symmetric monoidal structure, the category of presheaves $P_{\Sp_\mathbb{Q}} ({\cM^{\mathbb{Q}}})$ (resp. $P_{\Sp_{(p)}} ({\cM^{(p)}})$) acquires a canonical symmetric monoidal left $\Sp_\mathbb{Q}$-tensored (resp. $\Sp_{(p)}$-tensored) structure and we have a natural symmetric monoidal left $\Sp_\mathbb{Q}$-tensored (resp. $\Sp_{(p)}$-tensored) functor
$$P_{\Sp_\mathbb{Q}}({\cM^{\mathbb{Q}}})\xrightarrow{\sim}\NSp_\mathbb{Q} \quad \left(\mbox{resp. }, P_{\Sp_{(p)}}({\cM^{(p)}})\xrightarrow{\sim}\NSp_{(p)}\right)$$
which is an equivalence of the underlying $\infty$-categories.
\end{thm}
\subsubsection*{An explicit presentation of ${\cM^{\mathbb{Q}}}$ and $\NSp_\mathbb{Q} $}
We can use our results to give a very explicit description of  ${\cM^{\mathbb{Q}}}$, and therefore of the noncommutative rational stable homotopy category.

The rationalization functor $L^{\Sp}_{\mathbb{Q}}\colon \Sp \to \Sp_{\mathbb{Q}}$ is a symmetric monoidal left adjoint. As explained in Section~\ref{section: first stage}, we have an induced adjunction
$$(L^{\Sp}_{\mathbb{Q}})_! \colon \Cat^{\otimes}_{\Sp}  \leftrightarrows  \Cat^{\otimes}_{\Sp_\QQ} \noloc i_!.$$
Composing adjunctions we obtain
$$(L^{\Sp}_{\mathbb{Q}}\circ\Sigma^{\infty}_+)_! \colon \Cat^{\otimes}  \leftrightarrows  \Cat^{\otimes}_{\Sp_\QQ} \noloc (\Omega^{\infty}\circ i)_!.$$
We denote
$$\mathbb{P}\Inj_{\infty}^{\Sp_\QQ}:= (L^{\Sp}_{\mathbb{Q}}\circ\Sigma^{\infty}_+)_! (\mathbb{P}\Inj_{\infty}) \in \Cat^{\otimes}_{\Sp_\QQ}$$
and we denote the mate of $L^{\Sp}_{\QQ}\circ\Sigma^{\infty} \circ \widetilde{\End}\in \Map_{\Cat^{\otimes}} (\mathbb{P}\Inj_{\infty}^{\mathrm{op}}, \NSp_\QQ)$ under this adjunction by
$$\widetilde{\mathrm{E}}_{\QQ}\colon (\mathbb{P}\Inj_{\infty}^{\Sp_\QQ})^\mathrm{op} \to \NSp_\QQ \in \Cat^{\otimes}_{\Sp_\QQ}.$$



\begin{prop}\label{p:rat}
The functor
$$\widetilde{\mathrm{E}}_{\QQ}\colon (\mathbb{P}\Inj_{\infty}^{\Sp_\QQ})^\mathrm{op} \to \NSp_\QQ \in \Cat^{\otimes}_{\Sp_\QQ}$$
is fully faithful as an $\Sp_\QQ$-enriched functor  with essential image $$\cM^{\mathbb{Q}} \subseteq \NSp_\QQ.$$
\end{prop}
\begin{proof}
The statement about the  essential image is not effected by changing enrichment and is thus  clear from the description  of the functor $\End$.
The fully-faithfulness follows from lemma \ref{l:big} and corollary \ref{cor:rat}.
\end{proof}

In view of theorem \ref{theorem: rational main presentation}  and  proposition \ref{p:rat}.
We get the following result:
\begin{thm}\label{t:rational spectra}
We have a sequence of equivalences of symmetric monoidal $\infty$-categories $$\mathrm{Fun}(\mathbb{P}\Inj_{\infty}, \Sp_{\mathbb{\QQ}}) \cong P_{\Sp_\mathbb{Q}}((\mathbb{P}\Inj_{\infty}^{\Sp_\QQ})^{\mathrm{op}})
\cong \NSp_{\QQ}.$$
\end{thm}
\begin{proof}
The second equivalence   is an immediate corollary  of Theorem \ref{theorem: rational main presentation}  and  Proposition \ref{p:rat}. For the first equivalence note that
$$P_{\Sp_\mathbb{Q}}((\mathbb{P}\Inj_{\infty}^{\Sp_\QQ})^{\mathrm{op}}) = \mathrm{Fun}_{\Sp_\QQ}(\mathbb{P}\Inj_{\infty}^{\Sp_\QQ}, \Sp_{\mathbb{\QQ}}),$$
where $\mathrm{Fun}_{\Sp_\QQ}$ stands for $\Sp_\QQ$-enriched functors.
But we have an induced adjunction
$$(L^{\Sp}_{\mathbb{Q}}\circ\Sigma^{\infty}_+)_! \colon \Cat\leftrightarrows  \Cat_{\Sp_\QQ} \noloc (\Omega^{\infty}\circ i)_!,$$
so we have natural equivalences
$$\Fun_{\Sp_\QQ} (\mathbb{P}\Inj^{\Sp_\QQ}_{\infty},\Sp_\QQ)\simeq \Fun_{\Sp_\QQ}((L^{\Sp}_{\mathbb{Q}}\circ \Sigma^{\infty}_+)_! (\mathbb{P}\Inj_{\infty}),\Sp_\QQ)\simeq $$ $$\Fun(\mathbb{P}\Inj_{\infty}, (\Omega^{\infty}\circ i)_!\Sp_\QQ)\simeq \Fun(\mathbb{P}\Inj_{\infty}, \Sp_\QQ),$$
and are done.
\end{proof}

\subsubsection*{$p$-local and chromatic picture}
Now instead of rationalizing, suppose we fix a prime $p$ and localize everything at $p$. One can obtain further information about the $p$-localization of $\cM$. It follows from Proposition~\ref{prop: Lm facts intro} parts~\eqref{prime powers} and~\eqref{p-local} that the rank filtration of ${\cM^{(p)}}$ is  constant except at powers of $p$. Therefore it is natural to regrade the filtration of $\SSS^{k,l}$ as follows
\[
\SSS^{k,l,1} \hookrightarrow \SSS^{k,l,p} \hookrightarrow \SSS^{k,l,p^2}\hookrightarrow \cdots \hookrightarrow \SSS^{k,l,p^i}\cdots
\]
Let us loosely refer to $\SSS^{k,l,p^i}$ as morphisms of filtration $i$.
We can say that the $p$-localization of $\cM$ is a filtered category in the sense that the composition of a morphism of filtration $i$ and a morphism of filtration $j$ has filtration $i+j$. With this grading, (the $p$-localization) of $\cM$ is a graded category in the usual sense, that composition adds degrees.

Lastly, let us mention that the last part of Proposition~\ref{proposition: L_m facts} implies the following
\begin{cor}
Fix a prime $p$ and localize everything at $p$. The map
\[
\SSS^{k,l, p^n}\to \SSS^{k,l}
\]
induces an isomorphism on Morava $K(i)$-theory for $i\le n$.
\end{cor}
We wonder if one can use this lemma to say something interesting about the chromatic localization of $\NSp$, but we will not pursue this here.

\end{document}